\documentclass[11pt]{article}
\usepackage[tbtags]{amsmath}
\usepackage{amssymb}
\usepackage{amsthm}
\usepackage[misc]{ifsym}
\usepackage{cases}
\usepackage{mathrsfs}
\usepackage{graphicx}
\usepackage{subcaption} 
\usepackage{color}

\numberwithin{equation}{section}
\normalsize
\title{\bf Linear-Quadratic Mean Field Stackelberg Stochastic Differential Game with Partial Information and Common Noise \thanks{This work is supported by National Key R\&D Program of China (2022YFA1006104), National Natural Science Foundations of China (12471419, 12271304), and Shandong Provincial Natural Science Foundations (ZR2024ZD35, ZR2022JQ01).}}
\author{\normalsize  Yu Si\thanks{\it School of Mathematics, Shandong University, Jinan 250100, P.R. China, E-mail: 202112003@mail.sdu.edu.cn} , Jingtao Shi\thanks{\it Corresponding author. School of Mathematics, Shandong University, Jinan 250100, P.R. China, E-mail: shijingtao@sdu.edu.cn}}
\date{}
\newtheorem{mypro}{Proposition}[section]

\newtheorem{assumption}{Assumption}[section]
\newtheorem{remark}{Remark}[section]
\begin{document}

\maketitle

\noindent{\bf Abstract:}\quad This paper is concerned with a linear-quadratic mean field Stackelberg stochastic differential game with partial information and common noise, which contains a leader and a large number of followers. To be specific, the followers face a large population Nash game after the leader first announces his strategy, while the leader will then optimize his own cost functional on consideration of the followers' reactions. The state equation of the leader and followers are both general stochastic differential equations, where the diffusion terms contain both the control and state variables.  However, the followers' average state terms enter into the drift term of the leader's state equation, reflecting that the leader's state is influenced by the followers' states. By virtue of stochastic maximum principle with partial information and optimal filter technique, we deduce the open-loop adapted decentralized strategies and feedback decentralized strategies of this leader-followers system, and demonstrate that the decentralized strategies are the corresponding $\varepsilon$-Stackelberg-Nash equilibrium.

\vspace{2mm}


\vspace{2mm}


\section{Introduction}

Recently, the study of dynamic optimization in stochastic large-population systems has garnered significant attention. Distinguishing it from a standalone system, a large-population system comprises numerous agents, widely applied in fields such as engineering, finance and social science. In this context, the impact of a single agent is minimal and negligible, whereas the collective behaviors of the entire population are significant. All the agents are weakly coupled via the state average or empirical distribution in dynamics and cost functionals. Consequently, centralized strategies for a given agent, relying on information from all peers, are impractical. Instead, an effective strategy is to investigate the associated {\it mean-field games} (MFGs) to identify an approximate equilibrium by analyzing its limiting behavior. Along this research direction, we can obtain the decentralized strategies through the limiting auxiliary control problems and the related {\it consistency condition} (CC) system. The methodological foundations of MFGs, initially proposed by Lasry and Lions \cite{Lasry-Lions-2007} and independently by Huang et al. \cite{Huang-Caines-Malhame-2007, Huang-Malhame-Caines-2006}, have proven effective and tractable for analyzing weakly coupled stochastic controlled systems with mean field interactions, establishing approximate Nash equilibria. The interested readers can refer to can refer to \cite{Bardi-Priuli-2014, Du-Huang-Wu-2018, Hu-Huang-Li-2018, Hu-Huang-Nie-2018, Huang-Zhou-2020, Moon-Basar-2017, Wang-Zhang-2017} for {\it linear-quadratic} (LQ) MFGs, and refer to \cite{Bensoussan-Frehse-Yam-2013, Buckdahn-Djehiche-Li-Peng-2009, Carmona-Delarue-2013, Cong-Shi-2024, Huang-2010, Nguyen-Huang-2012} for further analysis of MFGs and related topics.

We mention that the above listed literature mainly focused on the case in which agents have access to full information. However, in the real world, usually agents can only get incomplete information at most cases. As far as I know, systems with incomplete information refer to those where the state cannot be fully observed by the agents. In mathematical terms, this means that the control variables are adapted to a smaller filteration than the full information. Specifically, there are mainly two cases: Firstly, the control variables are adapted to the information flow generated by some observation process, which is generally referred to as a partially observed problem; Secondly, the control variables are adapted to a sub-information flow generated by some Brownian motion, which is generally referred to as a partial information problem. Huang and Wang \cite{Huang-Wang-2016}  studied a class of dynamic optimization problems for a large-population system with partial information. Huang et al. \cite{Huang-Wang-Wu-2016} investigated a backward large-population system with partial information with full and partial information. Wang et al. \cite{Wang-Wu-Xiong-2018} studied the optimal control problems for {\it forward-backward stochastic differential equations} (FBSDEs) with incomplete information thoroughly.  \c Sen and Caines \cite{Sen-Caines-2016, Sen-Caines-2019} investigated the MFG problems with partial observation. Bensoussan et al.  \cite{Bensoussan-Feng-Huang-2021} focused on a class of linear-quadratic-Gaussian (LQG) MFGs with partial observation and common noise. Huang et al. \cite{Huang-Wang-Wang-Xiao-2024, Huang-Wang-Wang-Wang-2023} studied a forward-backward stochastic system MFG with partial observation and common noise and a backward stochastic system MFG with partial information and common noise, respectively.  Li et al. \cite{Li-Nie-Wu-2023} investigated a large-population problem with partial information. Readers can see \cite{Li-Nie-Wang-Yan-2023, Wang-Wang-Zhang-2020} and the references therein for further studies.

The Stackelberg differential game, also known as the leader-follower differential game, arises in markets where certain companies possess greater authority to dominate others or individual entities. von Stackelberg \cite{Stackelberg-1952} introduced the concept of the hierarchical solution. Yong \cite{Yong-2002} delved into a generalized framework of LQ stochastic leader-follower differential games.  Nourian et al. \cite{Nourian-Caines-Malhame-Huang-2012} studied a large-population LQ leader-follower stochastic multi-agent systems and established their ($\varepsilon_1,\varepsilon_2$)-Stackelberg-Nash equilibrium. Moon and Ba\c{s}ar in \cite{Moon-Basar-2018} investigated continuous-time mean field LQ Stackelberg stochastic differential games by the fixed-point method. Si and Wu \cite{Si-Wu-2021} explored a backward-forward LQ Stackelberg MFG, where the leader's state equation is backward, and the follower's state equation is forward. Feng and Wang \cite{Feng-Wang-2024} investigated an LQ social optimal problem with a leader and a amount of followers.  Wang \cite{Wang-2024} employed a direct method to solve LQ Stackelberg MFGs with a leader and a substantial number of followers.

In this paper, we consider a new class of LQ mean field Stackelberg stochastic differential games with partial information and common noise, which contains a leader and a number of followers. The
leader first announces his strategy and then each follower optimizes its own cost based on the leader's announcement. At last, the leader chooses his own optimal strategy based on the responses of the followers. Compared with the existing literatures, the contributions of this paper are listed as follows.
\begin{itemize}
  \item We introduced a general Stackelberg stochastic differential game with partial information and common noise. The common noise can not be observed by all agents, which force us to introduce a stochastic process as the frozen term, rather than a deterministic function, in the MFG analysis. Due to the framework of partial information, we need to use stochastic maximum principle with partial information and optimal filter technique to get the open loop decentralized optimal strategy and feedback decentralized optimal strategy of the auxiliary limiting problem.
  \item The state equation and cost functional of the leader both contain the state average term. This is motivated by real-world scenarios where all followers may have instantaneous and significant effects on the leader's dynamic. In addition, the state of the leader enters the cost functional of the followers, which means that changes in the leader can have an impact on the followers' evaluation criteria. This is consistent with the background where the leaders hold a dominant position. Therefore, complexity arises from the mean-field coupling term and highlight the strong coupling and interactivity among the state equations and cost functionals.
  \item The diffusion terms of all agents' state equations contain both the control and state variables, which makes the Hamiltonian system more complicated. And we apply decoupling technique to obtain the solvability of Hamiltonian system.
  \item Due to the application of dimension-extension techniques, the Riccati equations are asymmetric and complex, making their solvability quite challenging.
\end{itemize}

The rest of this paper is organized as follows. In Section 2, we formulate our problem. In Section 3, we introduce the auxiliary limiting problem and solve the problem of the followers and the leader in turn and derive the decentralized optimal strategies. In Section 4, we prove the decentralized optimal strategies are the $\varepsilon$-Stackelberg-Nash equilibria of the games. In Section 5, we provide a numerical example to show the effectiveness of our theoretical results. Finally, the conclusion is given in Section 6.

\section{Problem formulation}

Firstly, we introduce some notations that will be used throughout the paper. We consider a finite time interval $[0, T]$ for a fixed $T > 0$.
Let $\left(\Omega, \mathcal{F},  \left\{\mathcal{F}_t\right\}_{t\geq0},\mathbb{P}\right)$ be a complete filtered probability space, on which a standard $(N+1)$-dimensional Brownian motion $\left\{W(s), W_i(s), 1 \leq i \leq N\right\}_{0 \leq s \leq t}$ is defined. $\{\mathcal{F}_t\}$ is defined as the complete information of the system at time $t$. Then, for any $t \geq 0$, we have
$$
\mathcal{F}_t:=\sigma\left\{W_i(s), W_0(s), W(s) , 1 \leq i \leq N, 0 \leq s \leq t\right\}\vee \mathcal{N}_{\mathbb{P}},
$$
where $\mathcal{N}_{\mathbb{P}}$ is the class of all $\mathbb{P}$-null sets. Define the available information of the $i$-th follower $\mathcal{A}_i $ as
$$\mathcal{F}_t^{W_i}:=\sigma\left\{W_i(s), 0 \leq s \leq t\right\} \vee \mathcal{N}_{\mathbb{P}},$$
and the available information of the leader $\mathcal{A}_0 $ as
$$\mathcal{F}_t^{W_0}:=\sigma\left\{W_0(s), 0 \leq s \leq t\right\} \vee \mathcal{N}_{\mathbb{P}}.  $$

For $i=0,1,\ldots,N$, define the involved information of $\mathcal{A}_i $ as
$$\mathcal{F}_t^{W, W_i}:=\sigma\left\{W(s), W_i(s), 0 \leq s \leq t\right\} \vee \mathcal{N}_{\mathbb{P}},$$
and let the information that is not observable by everyone (common noise) as
$$\mathcal{F}_t^{W}:=\sigma\left\{W(s), 0 \leq s \leq t\right\} \vee \mathcal{N}_{\mathbb{P}}.$$
And define
$$
\begin{aligned}
 \mathcal{G}_t^l&:=\sigma\left\{W_i(s), 0\leq s\leq t, 0 \leq i \leq N\right\}\vee \mathcal{N}_{\mathbb{P}} ,\\
 \mathcal{G}_t^f&:=\sigma\left\{W_i(s), 0\leq s\leq t, 1 \leq i \leq N\right\}\vee \mathcal{N}_{\mathbb{P}} ,\\
 \mathcal{F}_t^{i}&:=\sigma\left\{W(s), W_0(s),  W_i(s), 0 \leq s \leq t\right\} \vee \mathcal{N}_{\mathbb{P}},\quad\mbox{for } i=1,\cdots,N.
\end{aligned}
$$

Let $\mathbb{R}^n$ be an $n$-dimensional Euclidean space with norm and inner product being defined as $|\cdot|$ and $\langle\cdot, \cdot\rangle$, respectively.
Next, we introduce three spaces. A bounded, measurable function $f(\cdot):[0, T] \rightarrow \mathbb{R}^n$ is denoted as $f(\cdot) \in L^{\infty}(0, T; \mathbb{R}^n)$. An $\mathbb{R}^n$-valued, $\mathbb{F}$-adapted stochastic process $f(\cdot): \Omega \times [0, T] \rightarrow \mathbb{R}^n$ satisfying $\mathbb{E} \int_0^T |f(t)|^2 dt < \infty$ is denoted as $f(\cdot) \in L_{\mathbb{F}}^2(0, T; \mathbb{R}^n)$. Similarly, an $\mathbb{R}^n$-valued, $\mathcal{F}_{T}$-measurable random variable $\xi$ with $\mathbb{E} \xi^2 < \infty$ is denoted as $\xi \in L_{\mathcal{F}_T}^2(\Omega, \mathbb{R}^n)$.

For any random variable or stochastic process $X$ and filtration $\mathcal{H}$, $\mathbb{E}X$ and $\mathbb{E}[X|\mathcal{H}]$ represent the mathematical expectation and conditional mathematical expectation of $X$, respectively. For a given vector or matrix \(M\), let \(M^{\top}\) represent its transpose. We denote the set of symmetric \(n \times n\) matrices (resp. positive semi-definite matrices) with real elements by \(\mathcal{S}^n\) (resp. \(\mathcal{S}_{+}^n\)). If \(M \in \mathcal{S}^n\) is positive (semi) definite, we abbreviate it as \(M > (\geq) 0\). For a positive constant \(k\), if \(M \in \mathcal{S}^n\) and \(M > kI\), we label it as \(M \gg 0\).

Now, let us focus on a comprehensive population system comprised of $N+1$ individual agents, denoted as $\left\{\mathcal{A}_i\right\}_{0 \leq i \leq N}$. The state $x_0(\cdot)\in \mathbb{R}$ of the leader $\mathcal{A}_0$ is given by the following linear $\mathrm{SDE}$
\begin{equation}\label{leader state}
\left\{\begin{aligned}
d x_0(t)= &\ \left[A_0 x_0(t)+B_0 u_0(t)+E_0 x^{(N)}(t)+b_0\right] d t  \\
& +\left[C_0 x_0(t)+D_0 u_0(t)+F_0 x^{(N)}(t)+\sigma_0\right] d W_0(t) \\
& +\left[\widetilde{C}_0 x_0(t)+\widetilde{D}_0 u_0(t)+\widetilde{F}_0 x^{(N)}(t)+\widetilde{\sigma}_0\right] d W(t) , \\
x_0(0)= &\ \xi_0,
\end{aligned}\right.
\end{equation}
where $\xi_0 \in \mathbb{R}$ represents the initial state, and $x^{(N)}(\cdot) := \frac{1}{N} \sum_{i=1}^N x_i(\cdot)$ signifies the average state of the followers. The corresponding coefficients $A_0(\cdot)$, $B_0(\cdot)$, $C_0(\cdot)$, $D_0(\cdot)$, $E_0(\cdot)$, $F_0(\cdot)$, $\widetilde{C}_0(\cdot)$, $\widetilde{D}_0(\cdot)$, $\widetilde{F}_0(\cdot)$, $b_0(\cdot)$, $\sigma_0(\cdot)$, $\widetilde{\sigma}_0(\cdot)$ are deterministic $\mathbb{R}$-valued functions.

The state $x_i(\cdot)$ of the follower $\mathcal{A}_i$ is given by the following linear $\mathrm{SDE}$
\begin{equation}\label{follower state}
\left\{\begin{aligned}
d x_i(t)  =&\ \left[A_1 x_i(t)+B_1 u_i(t)+E_1 x^{(N)}(t)+b_1\right] d t \\
& +\left[C_1 x_i(t)+D_1 u_i(t)+F_1 x^{(N)}(t)+\sigma_1\right] d W_i(t) \\
& +\left[\tilde{C}_1 x_i(t)+\widetilde{D}_1 u_i(t)+\widetilde{F}_1 x^{(N)}(t)+\widetilde{\sigma}_1\right] d W(t) , \\
x_i(0)  =&\ \xi,
\end{aligned}\right.
\end{equation}
where $\xi \in \mathbb{R}$ represents the initial state, the coefficients $A_1(\cdot)$, $B_1(\cdot)$, $C_1(\cdot)$, $D_1(\cdot)$, $E_1(\cdot)$, $F_1(\cdot)$, $\widetilde{C}_1(\cdot)$, $\widetilde{D}_1(\cdot)$, $\widetilde{F}_1(\cdot)$, $b_1(\cdot)$, $\sigma_1(\cdot)$, $\widetilde{\sigma}_1(\cdot)$ are deterministic $\mathbb{R}$-valued functions.

The decentralized admissible control set $\mathcal{U}^{l,d}_0$ of the leader $\mathcal{A}_0$ is defined as
$$
\mathcal{U}^{l,d}_0:=\left\{u_0(\cdot) \mid u_0(\cdot) \in L_{\mathcal{F}_t^{W_0}}^2\left(0, T ; \mathbb{R}\right)\right\},
$$
and the admissible centralized control set $\mathcal{U}_{0}^{l,c}$ of the leader $\mathcal{A}_0$ is defined as
$$
\mathcal{U}_0^{l,c}:=\left\{u_0(\cdot) \mid u_0(\cdot) \in L_{\mathcal{G}_t^l}^2\left(0, T ; \mathbb{R}\right)\right\} .
$$
For $i=1, \cdots, N$, the decentralized admissible control set $\mathcal{U}_{i}^{f,d}$ of the $i$-th follower $\mathcal{A}_i$ is defined as
$$
\mathcal{U}_{i}^{f,d}:=\left\{u_i(\cdot) \mid u_i(\cdot) \in L_{\mathcal{F}_t^{W_i}}^2\left(0, T ; \mathbb{R} \right)\right\},
$$
and the centralized admissible control set $\mathcal{U}_{i}^{f,c}$ of the $i$-th follower $\mathcal{A}_i$ is defined as
$$
\mathcal{U}_{i}^{f,c}:=\left\{u_i(\cdot) \mid u_i(\cdot) \in L_{\mathcal{G}_t^f}^2\left(0, T ; \mathbb{R}\right)\right\} .
$$

For the leader $\mathcal{A}_0$, the cost functional is defined by
\begin{equation}\label{leader cost}
\mathcal{J}_0\left(u_0(\cdot)\right)=\frac{1}{2} \mathbb{E} \int_0^T\left[R_0(t) u_0^2(t)+Q_0(t)\left(x_0(t)-x^{(N)}(t)\right)^2\right] d t+G_0 x_0(T)^2,
\end{equation}
where $Q_0(\cdot)$, $R_0(\cdot)$ are deterministic functions and $G_0$ is a constant. Let $u(\cdot)\equiv\left(u_1(\cdot), \ldots, u_N(\cdot)\right)$ be the set of control strategies of all followers and $u_{-i}(\cdot)\equiv\left(u_1(\cdot), \ldots, u_{i-1}(\cdot), u_{i+1}(\cdot), \ldots, u_N(\cdot)\right)$ be the set of control strategies except for $i$-th follower $\mathcal{A}_i$. For $i=1, \cdots, N$, the cost functional of the $i$-th follower $\mathcal{A}_i$ is defined by
\begin{equation}\label{follower cost}
\begin{aligned}
\mathcal{J}_i\left(u_i(\cdot), u_{-i}(\cdot)\right)
&=\frac{1}{2} \mathbb{E} \bigg\{\int_0^T \left[R_1(t) u_i^2(t)+Q_1(t)\left[x_i(t)-\left(\lambda x_0(t)+(1-\lambda) x^{(N)}(t)\right)\right]^2\right] dt\\
&\qquad\qquad + G_1 x_i(T)^2\bigg\},
\end{aligned}
\end{equation}
where $Q_1(\cdot)$, $R_1(\cdot)$ are deterministic functions and $G_1$ is a constant. Here $\lambda\in[0,1]$ is a constant.

Moreover, we introduce the following assumptions of coefficients.
\begin{assumption}\label{A2.1}
$A_0(\cdot),B_0(\cdot), E_0(\cdot),  C_0(\cdot), D_0(\cdot), F_0(\cdot),\widetilde{C}_0(\cdot), \widetilde{D}_0(\cdot), \widetilde{F}_0(\cdot), A_1(\cdot), B_1(\cdot), E_1(\cdot),\\  C_1(\cdot), D_1(\cdot), F_1(\cdot), \widetilde{C}_1(\cdot), \widetilde{D}_1(\cdot), \widetilde{F}_1(\cdot), b_0(\cdot), \sigma_0(\cdot), \widetilde{\sigma}_0(\cdot), b_1(\cdot), \sigma_1(\cdot), \widetilde{\sigma}_1(\cdot)\in L^{\infty}\left(0, T ; \mathbb{R}\right)$.
\end{assumption}
\begin{assumption}\label{A2.2}
$ Q_0(\cdot), Q_1(\cdot), R_0(\cdot), R_1(\cdot) \in L^{\infty}\left(0, T ; \mathbb{R}\right), Q_0(\cdot), Q_1(\cdot) \geq 0, R_0(\cdot), R_1(\cdot) \gg 0;\\ G_0(\cdot), G_1(\cdot) \in \mathbb{R}, G_0(\cdot), G_1(\cdot) \geq 0.$
\end{assumption}

We mention that under assumption \ref{A2.1}, the system of SDEs (\ref{leader state}), (\ref{follower state}) admits a unique solution $(x_0(\cdot),x_1(\cdot),\cdots,x_N(\cdot))\in L_{\mathcal{F}_t}^2(0, T; \mathbb{R}^{1+N})$. Under assumption \ref{A2.2}, the cost functionals (\ref{leader cost}), (\ref{follower cost}) are well-defined.

Our {\it LQ mean field Stackelberg stochastic differential game with partial information and common noise}, investigated in this paper, can be stated as follows.

\newtheorem{prob}{Problem}[section]
\begin{prob}\label{problem centralized}
Finding a set of strategy $u^*(\cdot) \equiv\left(u_0^*(\cdot), u_1^*(\cdot), \cdots, u_N^*(\cdot)\right)$ satisfying the following conditions:
	
(i) for given $u_0(\cdot) \in \mathcal{U}_{0}^{l,c}$, the control stategy $u_i^*(\cdot)$ is a mapping $u_i^*:\mathcal{U}_{0}^{l,c}\rightarrow\mathcal{U}_i^{f,c}$ satisfying
\begin{equation*}
\mathcal{J}_i\left(u^*_i[u_0](\cdot), u^*_{-i}[u_0](\cdot)\right)=\inf _{u_i(\cdot) \in\, \mathcal{U}_i^{f,c}} \mathcal{J}_i\left(u_i(\cdot), u^*_{-i}[u_0](\cdot)\right), \text{ for any } 1 \leq i \leq N;
\end{equation*}
where $u^*_{-i}[u_0](\cdot) \equiv\left(u_1^*[u_0](\cdot), \cdots, u_{i-1}^*[u_0](\cdot), u_{i+1}^*[u_0](\cdot), \cdots, u_N^*[u_0](\cdot)\right)$,

(ii) the control strategy  $u^*_0(\cdot)$ of the leader $\mathcal{A}_0$ satisfies
\begin{equation*}
\mathcal{J}_0\left(u^*_0(\cdot)\right)=\inf _{u_0(\cdot) \in\, \mathcal{U}_{0}^{l,c}} \mathcal{J}_0\left(u_0(\cdot)\right).
\end{equation*}
\end{prob}

We call $\left(u_0^*(\cdot), u_1^*(\cdot), \cdots, u_N^*(\cdot)\right)$ the {\it Stackelberg-Nash equilibrium} of Problem \ref{problem centralized}. Moreover, the corresponding state $\left(x_0^*(\cdot), x_1^*(\cdot), \cdots, x_N^*(\cdot)\right)$ is called the {\it optimal centralized trajectory}.

\section{Limiting Stackelberg-Nash equilibria}

Since the intricate nature arising from the coupling of the state-average $x^{(N)}(\cdot):= \frac{1}{N} \sum_{i=1}^N x_i(\cdot)$, Problem \ref{problem centralized} becomes challenging to investigate. We shall employ the MFG theory to seek an approximate Stackelberg-Nash equilibrium, which serves as a bridge between the ``centralized" LQ games and the limiting state-average as the number of agents $N$ approaches infinity. Typically, the state-average is replaced by its frozen limit term for computational convenience.

As $N\rightarrow+\infty$, let us assume that the state-average $x^{(N)}(\cdot)$ can be approximated by some $\mathcal{F}_t^W$-adapted processes $z(\cdot)$ that will be subsequently determined through a CC system.
For the leader $\mathcal{A}_0$, we introduce the following auxiliary state $\bar{x}_0(\cdot)\in L^2_{\mathcal{F}_t^{W, W_0}}(0,T;\mathbb{R})$ which satisfies the following linear SDE
\begin{equation}\label{leader limiting state}
\left\{\begin{aligned}
d \bar{x}_0(t) =&\ \left[A_0(t) \bar{x}_0(t)+B_0(t) u_0(t)+E_0(t) z(t)+b_0(t)\right] d t \\
& +\left[C_0(t) \bar{x}_0(t)+D_0(t) u_0(t)+F_0(t) z(t)+\sigma_0(t)\right] d W_0(t) \\
& +\left[\widetilde{C}_0(t) \bar{x}_0(t)+\widetilde{D}_0(t) u_0(t)+\widetilde{F}_0(t) z(t)+\widetilde{\sigma}_0(t)\right] d W(t) , \\
 \bar{x}_0(0) =&\ \xi_0,
\end{aligned}\right.
\end{equation}
and the limiting cost functional
\begin{equation}\label{leader limiting cost}
J_0\left(u_0(\cdot)\right)=\frac{1}{2} \mathbb{E} \int_0^T\left[R_0(t) u_0(t)^2+Q_0(t)\left(\bar{x}_0(t)-z(t)\right)^2\right] dt + G_0 \bar{x}_0^2(T).
\end{equation}

For every follower $\mathcal{A}_i$, we introduce the following auxiliary state $\bar{x}_i(\cdot)\in L^2_{\mathcal{F}_t^{W, W_i}}(0,T;\mathbb{R})$ which satisfies the following linear SDE
\begin{equation}\label{follower limiting state}
\left\{\begin{aligned}
d \bar{x}_i(t) =&\ \left[A_1(t) \bar{x}_i(t)+B_1(t) u_i(t)+E_1(t) z(t)+b_1(t)\right] d t \\
& +\left[C_1(t) \bar{x}_i(t)+D_1(t) u_i(t)+F_1(t) z(t)+\sigma_1(t)\right] d W_i(t) \\
& +\left[\widetilde{C}_1(t) \bar{x}_i(t)+\widetilde{D}_1(t) u_i(t)+\widetilde{F}_1 (t)z(t)+\widetilde{\sigma}_1(t)\right] d W(t) , \\
 \bar{x}_i(0) =&\ \xi,
\end{aligned}\right.
\end{equation}
and the limiting cost functional
\begin{equation}\label{follower limiting cost}
J_i\left(u_i(\cdot)\right)=\frac{1}{2} \mathbb{E} \int_0^T\left[R_1(t) u_i^2(t)+Q_1(t)\left[\bar{x}_i(t)-\left(\lambda \bar{x}_0(t)+(1-\lambda) z(t)\right)\right]^2\right] dt + G_1 \bar{x}_i^2(T).
\end{equation}

Subsequently, we propose the corresponding limiting Stackelberg-Nash stochastic differential game with partial information.
\begin{prob}\label{problem decentralized}
Finding a strategy $u^*_0(\cdot)$ for leader $\mathcal{A}_0$ and strategy $u^*_i(\cdot)$ for each followers $\mathcal{A}_i$ satisfying the following conditions:

(i) for given $u_0(\cdot) \in \mathcal{U}_{0}^{l,d}$, the control stategy $u_i^*$ is a mapping $u_i^*:\mathcal{U}_{0}^{l,d}\rightarrow\mathcal{U}_i^{f,d}$ satisfying
\begin{equation*}
J_i\left(u^*_i[u_0](\cdot)\right)=\inf _{u_i(\cdot) \in\, \mathcal{U}_i^{f,d}} J_i\left(u_i(\cdot)\right);
\end{equation*}

(ii) the control strategy  $u^*_0(\cdot)$ of the leader $\mathcal{A}_0$ satisfies
\begin{equation*}
J_0\left(u^*_0(\cdot)\right)=\inf _{u_0(\cdot) \in\, \mathcal{U}_0^{l,d}} J_0\left(u_0(\cdot)\right).
\end{equation*}
\end{prob}
The above $\left(u_0^*(\cdot), u_1^*(\cdot), u_2^*(\cdot), \cdots\right)$ is called the {\it decentralized Stackelberg-Nash equilibrium} of Problem \ref{problem decentralized}.

\subsection{Open-loop decentralized strategies of the followers}

Observing that the frozen limiting state-average simplifies Problem \ref{problem centralized} into essentially an LQ Stackelberg-Nash stochastic differential game with partial information, we will utilize the stochastic maximum principle with partial information to derive its Stackelberg-Nash equilibria. In the remainder of this section, for brevity and without causing ambiguity, we omit the time dependency of some functions and stochastic processes, except for the terminal time.

By applying the stochastic maximum principle with partial information (for example, Theorem 3.1 of Baghery and \O ksendal \cite{Baghery-Oksendal-2007}, or Proposition 2.1 of Shi et al. \cite{Shi-Wang-Xiong-2016}), we can obtain the open-loop decentralized optimal strategies for the subproblem (i) of Problem \ref{problem decentralized}.

\newtheorem{thm}{Theorem}[section]
\begin{thm}\label{open loop thm1 of follower}
Let Assumptions \ref{A2.1} and \ref{A2.2} hold. For $i=1,2, \cdots$ and any given $u_0(\cdot) \in \mathcal{U}_0^{l,d}$, suppose $\bar{x}_0(\cdot)$ is the solution to (\ref{leader limiting state}), then the open-loop decentralized optimal strategy of $i$-th follower $\mathcal{A}_i$ is given by
\begin{equation}\label{follower open loop optimal control}
u_i^*=R_1^{-1}\left(B_1 \hat{p}_i+D_1 \hat{q}_i+\widetilde{D}_1 \hat{\widetilde{q}}_i\right),\ \mathbb{P}\text{-a.s.},\ \text{a.e.}\ t\in[0,T],
\end{equation}
where we denote $\hat{p}_i:=\mathbb{E}\left[\left.p_i\right\rvert \mathcal{F}_t^{W_i}\right]$, $\hat{q}_i:=\mathbb{E}\left[\left.q_i\right\rvert \mathcal{F}_t^{W_i}\right]$, $\hat{\widetilde{q}}_i:=\mathbb{E}\left[\left.\widetilde{q}_i\right\rvert \mathcal{F}_t^{W_i}\right]$ for notational simplicity and $\left(\bar{x}_i^*(\cdot), p_i(\cdot), q_i(\cdot),\widetilde{q}_i(\cdot), q_{i,0}(\cdot)\right)\in L^2_{\mathcal{F}_t^{W, W_i}}(0,T;\mathbb{R})\times L_{\mathcal{F}^i}^2\left(0, T ; \mathbb{R}\right) \times L_{\mathcal{F}^i}^2\left(0, T ; \mathbb{R}\right) \times L_{\mathcal{F}^i}^2\left(0, T ; \mathbb{R}\right)$ satisfies the following stochastic Hamiltonian system
\begin{equation}\label{follower Hamiltonian system}
\left\{\begin{aligned}
 d \bar{x}_i^*=& \left[A_1 \bar{x}_i^*+R_1^{-1} B_1^2 \hat{p}_i+R_1^{-1} B_1 D_1 \hat{q}_i+R_1^{-1} B_1 \widetilde{D}_1 \hat{\widetilde{q}}_i+E_1 z+b_1\right] d t \\
& +\left[C_1 \bar{x}_i^*+R_1^{-1} D_1 B_1 \hat{p}_i+R_1^{-1} D_1^2 \hat{q}_i+R_1^{-1} D_1 \widetilde{D}_1 \hat{\widetilde{q}}_i+F_1 z+\sigma_1\right] d W_i \\
& +\left[\widetilde{C}_1 \bar{x}_i^*+R_1^{-1} \widetilde{D}_1 B_1 \hat{p}_i+R_1^{-1} D_1 \widetilde{D}_1 \hat{q}_i+R_1^{-1} \widetilde{D}_1^2 \hat{\widetilde{q}}_i+\widetilde{F}_1 z+\widetilde{\sigma}_1\right] dW, \\
d p_i=&-\left[A_1 p_i+C_1 q_i+\widetilde{C}_1 \widetilde{q}_i-Q_1\left[\bar{x}_i^*-\left(\lambda \bar{x}_0+\left(1-\lambda\right) z\right)\right]\right] d t\\
& +q_i d W_i+\widetilde{q}_i d W+q_{i,0} d W_0 ,\\
 \bar{x}_i^*(0)=&\ \xi,\quad p_i(T)=-G_1\bar{x}^*_i(T),
\end{aligned}\right.
\end{equation}
with $z(\cdot)\in L^2_{\mathcal{F}_.^W}(0,T;\mathbb{R})$ to be determined.
\end{thm}
\begin{proof}
For each $i=1,2, \cdots$, for the open-loop decentralized optimal strategy $u_i^*(\cdot)$, assuming $\bar{x}_i^*(\cdot)$ represents the corresponding state trajectory, and $\left(p_i(\cdot), q_i(\cdot),\widetilde{q}_i(\cdot), q_{i,0}(\cdot)\right)$ is the unique solution to the second equation in (\ref{follower Hamiltonian system}) relative to $\left(\bar{x}_i^*(\cdot), u_i^*(\cdot)\right)$, the stochastic maximum principle with partial information can be expressed in the following form
$$
\mathbb{E}\left[\left(B_1 p_i+D_1 q_i+\widetilde{D}_1 \widetilde{q}_i-R_1 u_i^* \right)\left(u_i-u_i^*\right)\rvert \mathcal{F}_t^{W_i}\right] =0 , \text{for any } u_i \in \mathbb{R},\ \mathbb{P}\text{-a.s.},\ \text{a.e.}\ t\in[0,T].
$$
Then we can derive that
$$
u_i^*=R^{-1}\left(B_1 \hat{p}_i+D_1 \hat{q}_i+\widetilde{D}_i \hat{\widetilde{q}}_i\right),\ \mathbb{P}\text{-a.s.},\ \text{a.e.}\ t\in[0,T].
$$
The proof is complete.
\end{proof}

\begin{remark}
	The adapted solution to the BSDE in (\ref{follower Hamiltonian system}) is adapted stochastic processes $(p_i(\cdot),\\ q_i(\cdot),\widetilde{q}_i(\cdot), q_{i,0}(\cdot))$, where $\left(q_i(\cdot),\widetilde{q}_i(\cdot), q_{i,0}(\cdot)\right)$ is to ensure the adaptability of the BSDE (See \cite{Ma-Yong-1999}).
\end{remark}

Moreover, since the cost functional (\ref{follower limiting cost}) of Problem \ref{problem decentralized} (i) is strictly convex, it allows a unique optimal control, so the sufficiency of the optimal control (\ref{follower open loop optimal control}) can also be obtained. In addition, the state feedback representation of (\ref{follower open loop optimal control}) can be obtained by Riccati equations.

Next, we will study the unknown frozen limiting state-average $z(\cdot)$, which is an $\mathcal{F}_t^W$-adapted process. When $N \rightarrow \infty$, we would like to approximate $x_i^*(\cdot)$ by $\bar{x}_i^*(\cdot)$, thus $\frac{1}{N} \sum_{i=1}^N x^*_i(\cdot)$ is approximated by $\frac{1}{N} \sum_{i=1}^N \bar{x}_i^*(\cdot)$. Since the $\mathcal{F}_t^{W, W_i}$-adapted process $\bar{x}_i^*(\cdot)$ and $\mathcal{F}_t^{W, W_j}$-adapted process $\bar{x}_j^*(\cdot)$ (for $i\neq j$ and $i,j=1,2,\cdots$) are identically distributed and conditionally independent given $\mathbb{E}\left[\cdot \mid \mathcal{F}^{W}\right]$, we can apply the conditional strong law of large numbers (Majerek et al. \cite{Majerek-Nowak-Zieba-2005}) to draw a conclusion
\begin{equation}\label{SLLN}
z(\cdot)=\lim _{N \rightarrow \infty} \frac{1}{N} \sum_{i=1}^N \bar{x}_i^*(\cdot)=\mathbb{E}\left[\bar{x}_i^*(\cdot) \mid \mathcal{F}_{.}^{W}\right].
\end{equation}

According to (\ref{follower limiting state}), by replacing $z(\cdot)$ by $\mathbb{E}\left[\bar{x}_i^*(\cdot) \mid \mathcal{F}_{.}^{W}\right]$, we can obtain for each $i=1,2, \cdots$,
\begin{equation}\label{follower explicit state}
\left\{\begin{aligned}
d \bar{x}_i^* =&\ \left[A_1 \bar{x}_i^*+B_1 u_i^*+E_1 \mathbb{E}\left[\bar{x}_i^* \mid \mathcal{F}_{t}^{W}\right]+b_1\right] d t \\
& +\left[C_1 \bar{x}_i^*+D_1 u_i^*+F_1 \mathbb{E}\left[\bar{x}_i^* \mid \mathcal{F}_{t}^{W}\right]+\sigma_1\right] d W_i \\
& +\left[\widetilde{C}_1 \bar{x}_i^*+\widetilde{D}_1 u_i^*+\widetilde{F}_1 \mathbb{E}\left[\bar{x}_i^* \mid \mathcal{F}_{t}^{W}\right]+\widetilde{\sigma}_1\right] d W , \\
 \bar{x}_i^*(0) =&\ \xi.
\end{aligned}\right.
\end{equation}
Taking expectation on both sides of the equation (\ref{follower explicit state}), we have

\begin{equation}\label{follower limiting E state}
\left\{\begin{aligned}
d \mathbb{E} \bar{x}_i^*=&\ \left[\left(A_1+E_1\right) \mathbb{E} \bar{x}_i^*+B_1 \mathbb{E} u_i^* +b_1\right] d t, \\
 \mathbb{E} \bar{x}_i^*(0)=&\ \xi,\quad\mbox{for }i=1, \cdots, N.
\end{aligned}\right.
\end{equation}
Since $u_i^*(\cdot)$ is $\mathcal{F}_t^{W_i}$-adapted, we can obtain $\mathbb{E}\left[u_i^* \mid \mathcal{F}_t^{W}\right]=\mathbb{E}\left[u_i^*\right]$, for  $i=1,2, \cdots$. Then taking $\mathbb{E}\left[\cdot \mid \mathcal{F}_{.}^{W}\right]$ on both sides of (\ref{follower explicit state}), we get
\begin{equation}\label{z state}
\left\{\begin{aligned}
d z=&\ \left[\left(A_1+E_1\right) z+B_1 \mathbb{E} u_i^*+b_1\right] d t+\left[\left(\widetilde{C}_1+\widetilde{F}_1\right) z+\widetilde{D}_1 \mathbb{E} u_i^*+\widetilde{\sigma}_1\right] d W, \\
z(0)=&\ \xi,
\end{aligned}\right.
\end{equation}
which admits a unique solution $z(\cdot)\in L^2_{\mathcal{F}_t^W}(0,T;\mathbb{R})$. This indicates that
\begin{equation}\label{Ez state}
\left\{\begin{aligned}
d \mathbb{E}z=&\ \left[\left(A_1+E_1\right) \mathbb{E}z+B_1 \mathbb{E} u_i^*+b_1\right] d t ,\\
\mathbb{E}z(0)=&\ \xi.
\end{aligned}\right.
\end{equation}
That is to say, $\mathbb{E}\bar{x}_i^*(t)\equiv\mathbb{E}z(t),t\in[0,T]$, for  $i=1,2, \cdots$.

By noticing the second equation of (\ref{follower Hamiltonian system}) is coupled with the state equation (\ref{leader limiting state}) of leader $\mathcal{A}_0$ and replacing $z(\cdot)$ by $\mathbb{E}\left(\bar{x}^*_i(\cdot)\mid\mathcal{F}^{W}_\cdot\right)$, then we derive the following CC system, which is a {\it conditional mean field FBSDE} (CMF-FBSDE) of $\left(\bar{x}_i^*(\cdot),\bar{x}_0(\cdot), p_i(\cdot), q_i(\cdot),\widetilde{q}_i(\cdot), q_{i,0}(\cdot)\right)\in L^2_{\mathcal{F}_t^{W, W_i}}(0,T;\mathbb{R})\times L^2_{\mathcal{F}_t^{W, W_0}}(0,T;\mathbb{R})\times L_{\mathcal{F}^i}^2\left(0, T ; \mathbb{R}\right) \times L_{\mathcal{F}^i}^2\left(0, T ; \mathbb{R}\right) \times L_{\mathcal{F}^i}^2\left(0, T ; \mathbb{R}\right)$ (\cite{Shi-Wang-Xiong-2016}):
\begin{equation}\label{CC system}
\left\{\begin{aligned}
 d \bar{x}_i^*=&\ \left\{A_1 \bar{x}_i^*+R_1^{-1} B_1^2 \hat{p}_i+R_1^{-1} B_1 D_1 \hat{q}_i+R_1^{-1} B_1 \widetilde{D}_1 \hat{\widetilde{q}}_i+E_1 \mathbb{E}\left[\bar{x}_i^* \mid \mathcal{F}_t^{W}\right]
+b_1\right\} d t \\
& +\left[C_1 \bar{x}_i^*+R_1^{-1} D_1 B_1 \hat{p}_i+R_1^{-1} D_1^2 \hat{q}_i+R_1^{-1} D_1 \widetilde{D}_1 \hat{\widetilde{q}}_i+F_1 \mathbb{E}\left[\bar{x}_i^* \mid \mathcal{F}_t^{W}\right]
+\sigma_1\right] d W_i \\
& +\left[\widetilde{C}_1 \bar{x}_i^*+R_1^{-1} \widetilde{D}_1 B_1 \hat{p}_i+R_1^{-1} D_1 \widetilde{D}_1 \hat{q}_i+R_1^{-1} \widetilde{D}_1^2 \hat{\widetilde{q}}_i+\widetilde{F}_1 \mathbb{E}\left[\bar{x}_i^* \mid \mathcal{F}_t^{W}\right]
+\widetilde{\sigma}_1\right] dW, \\
d \bar{x}_0 =&\ \left[A_0 \bar{x}_0+B_0 u_0+E_0 \mathbb{E}\left[\bar{x}_i^* \mid \mathcal{F}_t^{W}\right]+b_0\right] d t \\
& +\left[C_0 \bar{x}_0+D_0 u_0+F_0 \mathbb{E}\left[\bar{x}_i^* \mid \mathcal{F}_t^{W}\right]+\sigma_0\right] d W_0 \\
& +\left[\widetilde{C}_0 \bar{x}_0+\widetilde{D}_0 u_0+\widetilde{F}_0 \mathbb{E}\left[\bar{x}_i^* \mid \mathcal{F}_t^{W}\right]+\widetilde{\sigma}_0\right] d W , \\
d p_i=&-\left[A_1 p_i+C_1 q_i+\widetilde{C}_1 \widetilde{q}_i-Q_1\left[\bar{x}_i^*-\left(\lambda \bar{x}_0+\left(1-\lambda\right) \mathbb{E}\left[\bar{x}_i^* \mid \mathcal{F}_t^{W}\right]
\right)\right]\right] d t\\
&+q_i d W_i+\widetilde{q}_i d W+q_{i,0} d W_0 ,\\
 \bar{x}_i^*(0)=&\ \xi,\quad\bar{x}_0(0) =\ \xi_0,\quad p_i(T)=-G_1\bar{x}_i^*(T),\quad\mbox{for }i=1, \cdots, N.
\end{aligned}\right.
\end{equation}
By noting that (\ref{CC system}) possesses a fully coupled structure and contains conditional expectation terms, we will discuss its well-posedness in the following subsection.

\subsection{State feedback decentralized strategies of the followers}

In this section, we derive the state feedback representation of the decentralized optimal strategies (\ref{follower open loop optimal control}) of the followers, through Riccati equations.

Noting the terminal condition and structure of (\ref{CC system}), for each $i=1,2, \cdots$, we suppose
\begin{equation}\label{follower decouple form}
p_i(t)=-P_1(t) \bar{x}_i^*(t)-P_2(t) \mathbb{E}\left[\bar{x}_i^*(t)\right]-\varphi_i(t),\quad t\in[0,T],
\end{equation}
with $P_1(T)=G_1$, $P_2(T)=0$ for two deterministic differentiable functions $P_1(\cdot),P_2(\cdot)$, and with $\varphi_i(T)=0$, for an $\mathcal{F}^i_t$-adapted process triple $(\varphi_i(\cdot),\xi_{i,0}(\cdot),\xi_i(\cdot))$ satisfying a BSDE
$$
d \varphi_i{(t)}=\alpha_i(t) d t+\xi_{i,0} d W_0+\xi_i d W_i,
$$
where the $\mathcal{F}^i_t$-adapted process $\alpha_i(\cdot)$ is to be determined. Applying It\^o's formula to (\ref{follower decouple form}) and
comparing the coefficients of the diffusion terms, we achieve
\begin{equation}\label{32}
\left\{\begin{aligned}
q_i=&-P_1\left[C_1 \bar{x}_i^*+ D_1 u^*_i +F_1 \mathbb{E}\left[\bar{x}_i^* \mid \mathcal{F}_t^{W}\right]+\sigma_1\right]-\xi_i,\\
\widetilde{q}_i=&-P_1\left[\widetilde{C}_1 \bar{x}_i^*+\widetilde{D}_1 u^*_i +\widetilde{F}_1 \mathbb{E}\left[\bar{x}_i^* \mid \mathcal{F}_t^{W}\right]+\widetilde{\sigma}_1\right],\\
q_{i,0}=&-\xi_{i,0},\,\quad \mbox{for }i=1,2, \cdots.
\end{aligned}\right.
\end{equation}
Then, taking the conditional expectation $\mathbb{E}[\cdot \mid \mathcal{F}_t^{W_i}]$, we can obtain
\begin{equation}\label{33}
\left\{\begin{aligned}
 \hat{p}_i=&-P_1 \hat{\bar{x}}_i^*-P_2 \mathbb{E}\bar{x}_i^*-\hat{ \varphi}_i ,\\
 \hat{q}_i=&-P_1\left[C_1 \hat{\bar{x}}_i^*+D_1 u^*_i +F_1 \mathbb{E}\bar{x}_i^*+\sigma_1\right]-\hat{\xi}_i,\\
 \hat{\widetilde{q}}_i=&-P_1\left[\widetilde{C}_1 \hat{\bar{x}}_i^*+\widetilde{D}_1 u^*_i+\widetilde{F}_1 \mathbb{E}\bar{x}_i^*+\widetilde{\sigma}_1\right], \\
 \hat{q}_{i,0}=&- \hat{\xi}_{i,0},\quad\mbox{for }i=1,2, \cdots.
\end{aligned}\right.
\end{equation}
Substituting them into (\ref{follower open loop optimal control}), we have
\begin{equation}\label{follower feedback strategy}
\begin{aligned}
u_i^* = &-\mathcal{R}_1^{-1}\left[ \left(B_1 + D_1 C_1 + \widetilde{D}_1 \widetilde{C}_1\right) P_1 \hat{\bar{x}}_i^* + \left(B_1 P_2 + D_1 F_1 P_1 + \widetilde{D}_1 \widetilde{F}_1 P_1\right) \mathbb{E}\bar{x}_i^* \right. \\
& \left. + B_1 \hat{\varphi}_i + D_1 P_1 \sigma_1 + \widetilde{D}_1 P_1 \widetilde{\sigma}_1 -D_1\hat{\xi}_i\right],\ \mathbb{P}\text{-a.s.},\ \text{a.e.}\ t\in[0,T],\quad\mbox{for }i=1,2, \cdots.
\end{aligned}
\end{equation}
where $\mathcal{R}_1:=R_1 + D_1^2 P_1  + \widetilde{D}_1^2 P_1 $, and
\begin{equation}\label{follower feedback E strategy}
\begin{aligned}
\mathbb{E} u_i^*=&-\mathcal{R}_1^{-1}\left\{\left[\left(B_1 + D_1 C_1 + \widetilde{D}_1 \widetilde{C}_1\right) P_1+\left(B_1 P_2+D_1 F_1 P_1+\widetilde{D}_1 \widetilde{F}_1 P_1\right)\right]\mathbb{E}\bar{x}_i^*\right. \\
&\left.+B_1 \mathbb{E} \varphi_i+\left(D_1 \sigma_1+\widetilde{D}_1 \widetilde{\sigma}_1\right) P_1-D_1 \mathbb{E}\xi_i\right\},\quad\mbox{for }i=1,2, \cdots.
\end{aligned}
\end{equation}
By compare the coefficients of the drift terms, we get for $i=1,2, \cdots$,
$$
\begin{aligned}
-\alpha_i = & \left(\dot{P}_1 + (2 A_1 + C_1^2 + \widetilde{C}_1^2) P_1 + Q_1\right) \bar{x}_i^* + \left(\dot{P}_2 + \left(2 A_1 + E_1\right) P_2\right) \mathbb{E}\bar{x}_i^* \\
& + \left(P_1(E_1 + C_1 F_1 + \widetilde{C}_1 \widetilde{F}_1) - Q_1(1 - \lambda)\right) \mathbb{E}\left[\bar{x}_i^* \mid \mathcal{F}_t^{W}\right]  - \left(B_1 + C_1 D_1 + \widetilde{C}_1 \widetilde{D}_1\right) P_1 \mathcal{R}_1^{-1} \\
&  \left[\left(B_1 + D_1 C_1 + \widetilde{D}_1 \widetilde{C}_1\right) P_1 \hat{\bar{x}}_i^* + \left(B_1 P_2 + D_1 F_1 P_1 + \widetilde{D}_1 \widetilde{F}_1 P_1\right) \mathbb{E}\bar{x}_i^* \right. \\
& \left. + B_1 \hat{\varphi}_i + D_1 P_1 \sigma_1 + \widetilde{D}_1 P_1 \widetilde{\sigma}_1-D_1\hat{\xi}_i\right]  - P_2 B_1 \mathcal{R}_1^{-1} \left\{\left[\left(B_1 + D_1 C_1 + \widetilde{D}_1 \widetilde{C}_1\right) P_1  \right. \right.\\
& \left.\left. + B_1 P_2 + D_1 F_1 P_1 + \widetilde{D}_1 \widetilde{F}_1 P_1\right] \mathbb{E}\bar{x}_i^* + B_1 \mathbb{E}\varphi_i + \left(D_1 \sigma_1 + \widetilde{D}_1 \widetilde{\sigma}_1\right) P_1-D_1\mathbb{E}\xi_i\right\} \\
& + P_1 b_1 + P_2 b_1 + A_1 \varphi_i + C_1 P_1 \sigma_1 +C_1\xi_i+ \widetilde{C}_1 P_1 \widetilde{\sigma}_1 - Q_1 \lambda \bar{x}_0,
\end{aligned}
$$
Noting $\bar{x}_0(\cdot)$ is $\mathcal{F}^{W,W_0}_t$-adapted, we have $\mathbb{E}\left[\bar{x}_0 \mid \mathcal{F}_t^{W_i}\right]=\mathbb{E} \bar{x}_0 $. Then take the conditional expectation $\mathbb{E}\left[\cdot \mid \mathcal{F}^{W_i}_{\cdot}\right]$ on both sides of above equation, we can obtain
$$
\begin{aligned}
 -\hat{\alpha}_i=&\left[\dot{P}_1+\left(2 A_1+C_1^2+\widetilde{C}_1^2\right) P_1-\mathcal{R}_1^{-1}\left(B_1+C_1 D_1+\widetilde{C}_1 \widetilde{D}_1\right)^2 P_1^2 +Q_1\right] \hat{\bar{x}}_i^* \\
& +\left\{\dot{P}_2+\left(2 A_1+E_1\right) P_2+P_1\left(E_1+C_1 F_1+\widetilde{C}_1 \widetilde{F}_1\right)-Q_1(1-\lambda)\right.\\
&\quad-\left(B_1+C_1 D_1+\widetilde{C}_1 \widetilde{D}_1\right)P_1 \mathcal{R}_1^{-1}\left(B_1 P_2+D_1 F_1 P_1+\widetilde{D}_1 \widetilde{F}_1 P_1\right) \\
&\quad\left.-P_2 B_1 \mathcal{R}_1^{-1}\left[\left(B_1+D_1 C_1+\widetilde{D}_1 \widetilde{C}_1\right) P_1+B_1 P_2+D_1 F_1 P_1+\widetilde{D}_1 \widetilde{F}_1 P_1\right]\right\}\mathbb{E}\bar{x}_i^*\\
& -\left(B_1+C_1 D_1+\widetilde{C}_1 \widetilde{D}_1\right) P_1 \mathcal{R}_1^{-1}\left(B_1  \hat{\varphi}_i+D_1 P_1 \sigma_1+\widetilde{D}_1 P_1 \widetilde{\sigma}_1-D_1\hat{\xi}_i\right) \\
& -P_2 B_1 \mathcal{R}_1^{-1}\left[B_1 \mathbb{E} \varphi_i+\left(D_1 \sigma_1+\widetilde{D}_1 \widetilde{\sigma}_1\right) P_1-D_1\mathbb{E}\xi_i\right] \\
& +P_1 b_1+P_2 b_1+A_1 \hat{\varphi}_i+C_1 P_1 \sigma_1+C_1\xi_i+\widetilde{C}_1 P_1 \widetilde{\sigma}_1-Q_1 \lambda \mathbb{E}\bar{x}_0,\quad\mbox{for }i=1,2, \cdots.
\end{aligned}
$$

Thus, we introduce the Riccati equations as follows:
\begin{equation}\label{RE P1}
\left\{\begin{array}{l}
\dot{P}_1+\left(2 A_1+C_1^2+\widetilde{C}_1^2\right) P_1-\mathcal{R}_1^{-1}\left(B_1+C_1 D_1+\widetilde{C}_1 \widetilde{D}_1\right)^2 P_1^2+Q_1=0, \\
P_1(T)=G_1,
\end{array}\right.
\end{equation}
\begin{equation}\label{RE P2}
\left\{\begin{aligned}
& \dot{P}_2+P_2\left[2 A_1+E_1-2\left(B_1+C_1 D_1+\widetilde{C}_1 \widetilde{D}_1\right) P_1 \mathcal{R}_1^{-1} B_1-\left(D_1 F_1+\widetilde{D}_1 \widetilde{F}_1\right) B_1 P_1\right] \\
& -B_1^2 \mathcal{R}_1^{-1} P_2^2 +P_1\left(E_1+C_1 F_1+\widetilde{C}_1 \widetilde{F}_1\right)\\
& -\left(B_1+C_1 D_1+\widetilde{C}_1 \widetilde{D}_1\right) P_1^2 \mathcal{R}_1^{-1}\left(D_1 F_1+\widetilde{D}_1 \widetilde{F}_1\right)-Q_1(1-\lambda)=0, \\
& P_2(T)=0,
\end{aligned}\right.
\end{equation}
and the equation of $(\hat{\varphi}_i(\cdot),\hat{\xi}_i(\cdot))\in L^2_{\mathcal{F}_t^{W_i}}(0,T;\mathbb{R})\times L^2_{\mathcal{F}_t^{W_i}}(0,T;\mathbb{R})$ as
\begin{equation}\label{equation of hat varphi}
\left\{\begin{aligned}
d \hat{\varphi}_i= & \left\{\left(B_1+C_1 D_1+\widetilde{C}_1 \widetilde{D}_1\right) P_1 \mathcal{R}_1^{-1}\left(B_1 \hat{\varphi}_i+D_1 P_1 \sigma_1+\widetilde{D}_1 P_1 \widetilde{\sigma}_1\right)\right. \\
& +P_2 B_1 \mathcal{R}_1^{-1}\left[B_1 \mathbb{E}\varphi_i+\left(D_1 \sigma_1+\widetilde{D}_1 \widetilde{\sigma}_1 \right)P_1-D_1 \mathbb{E}\xi_i \right]\\
&\left.-\left[P_1 b_1+P_2 b_1+A_1 \hat{\varphi}_i+C_1 P_1 \sigma_1 +C_1\hat{\xi}_i +\widetilde{C}_1 P_1 \widetilde{\sigma}_1-Q_1 \lambda \mathbb{E}\bar{x}_0\right]\right\} d t+\hat{\xi}_i dW_i,\\
 \hat{\varphi}_i(T)=&\ 0,\quad\mbox{for }i=1,2, \cdots.
\end{aligned}\right.
\end{equation}
By the existence and uniqueness of the solution to mean-field type BSDEs (Li et al. \cite{Li-Sun-Xiong-2019}), in fact, (\ref{equation of hat varphi}) degenerates to the following linear {\it backward ODE} (BODE):
\begin{equation}\label{equation of E varphi}
\left\{\begin{aligned}
d \mathbb{E} \varphi_i=&\left(M_1 \mathbb{E} \varphi_i+Q_1\lambda \mathbb{E} \bar{x}_0+M_2\right) d t, \\
\mathbb{E} \varphi_i(T)=&\ 0,\quad\mbox{for }i=1,2, \cdots,
\end{aligned}\right.
\end{equation}
which is coupled with the linear ODE of $E\bar{x}_0$, where we have denoted
$$
\left\{\begin{aligned}
M_1:=&\left(B_1+C_1 D_1+\widetilde{C}_1 \widetilde{D}_1\right) P_1 \mathcal{R}_1^{-1} B_1-A_1 +P_2 B_1 \mathcal{R}_1^{-1} B_1,\\
M_2:=&\left(B_1+C_1 D_1+\widetilde{C}_1 \widetilde{D}_1\right) P_1 \mathcal{R}_1^{-1}\left(D_1 P_1 \sigma_1+\widetilde{D}_1 P_1 \widetilde{\sigma}_1\right)+P_2 B_1 \mathcal{R}_1^{-1}\left(D_1 \sigma_1+\widetilde{D}_1 \widetilde{\sigma}_1\right) P_1\\
&-P_1 b_1-P_2 b_1-C_1 P_1 \sigma_1-\widetilde{C}_1 P_1 \widetilde{\sigma}_1.
\end{aligned}\right.
$$
Noting (\ref{equation of E varphi}) is independent of $i$, we could drop the subscript $i$ in the rest part of this paper, that is, $\varphi_i(t)\equiv\varphi(t),t\in[0,T]$, for all $i=1,2, \cdots$. Therefore, we can get the following {\it forward-backward ODEs} (FBODEs):
\begin{equation}\label{FBODE}
\left\{\begin{aligned}
d \mathbb{E}z=&\left(N_1 \mathbb{E}z -B_1 \mathcal{R}_1^{-1}B_1 \mathbb{E} \varphi+ N_2\right)  d t, \\
d \mathbb{E}\bar{x}_0 =&\left(A_0 \mathbb{E}\bar{x}_0+B_0 \mathbb{E}u_0+E_0 \mathbb{E}z +b_0\right) d t, \\
d \mathbb{E} \varphi=&\left(M_1 \mathbb{E} \varphi+Q_1\lambda \mathbb{E} \bar{x}_0+M_2\right) d t, \\
 \mathbb{E} z(0)=&\ \xi,\ \mathbb{E}\bar{x}_0(0) = \xi_0,\ \mathbb{E} \varphi(T)=0,
\end{aligned}\right.
\end{equation}
where
$$
\left\{\begin{aligned}
N_1:=&\ A_1+E_1-B_1 \mathcal{R}_1^{-1}\left[\left(B_1 + D_1 C_1 + \widetilde{D}_1 \widetilde{C}_1\right) P_1+B_1 P_2+D_1 F_1 P_1+\widetilde{D}_1 \widetilde{F}_1 P_1\right],\\
N_2:=&-B_1 \mathcal{R}_1^{-1}\left(D_1 \sigma_1+\widetilde{D}_1 \widetilde{\sigma}_1\right) P_1 +b_1.
\end{aligned}\right.
$$

Now, we derive the solvability of the equations (\ref{RE P1}), (\ref{RE P2}) and (\ref{FBODE}). Firstly, we introduce
\begin{assumption}\label{A3.1}
$$
\begin{aligned}
P_1\left(E_1+C_1 F_1+\widetilde{C}_1 \widetilde{F}_1\right)-\left(B_1+C_1 D_1+\widetilde{C}_1 \widetilde{D}_1\right) P_1^2 \mathcal{R}_1^{-1}\left(D_1 F_1+\widetilde{D}_1 \widetilde{F}_1\right)-Q_1(1-\lambda)\geq 0.
\end{aligned}
$$
\end{assumption}

Under assumptions \ref{A2.1} and \ref{A2.2}, (\ref{RE P1}) is standard Riccati equation. Noticing that (\ref{RE P2}) is a standard Riccati equation under assumption \ref{A3.1}. Therefore, applying the standard results of Chapter 6 in Yong and Zhou \cite{Yong-Zhou-1999}, both (\ref{RE P1}) and (\ref{RE P2}) exist unique solutions.

Then, we derive the condition under which (\ref{FBODE}) is solvable. To this end, we first rewrite (\ref{FBODE}) in the following form:
\begin{equation}\label{dimensional expansion of FBODE}
\left\{\begin{aligned}
&d\left(\begin{array}{c}
\mathbb{E}z \\
\mathbb{E}\bar{x}_0  \\
\mathbb{E}\varphi
\end{array}\right)=\Pi\left(\begin{array}{c}
\mathbb{E}z \\
\mathbb{E}\bar{x}_0  \\
\mathbb{E}\varphi
\end{array}\right)+\Delta, \\
&\mathbb{E}z(0)=\xi, \quad \mathbb{E}\bar{x}_0 (0)=\xi_0, \quad \mathbb{E}\varphi (T)=0,
\end{aligned}\right.
\end{equation}
where
$$
\Pi:=\left(\begin{array}{ccc}
N_1 & 0 & -B_1 \mathcal{R}^{-1} B_1 \\
E_0 & A_0 & 0 \\
0 & Q_1\lambda & M_1
\end{array}\right), \quad \Delta:=\left(\begin{array}{c}
N_2 \\
B_0\mathbb{E}u_0+b_0 \\
M_2
\end{array}\right) .
$$

Then, by the variation of constant formula, we have
$$
\left(\begin{array}{c}
\mathbb{E}z \\
\mathbb{E}\bar{x}_0  \\
\mathbb{E}\varphi
\end{array}\right)=\Theta(t)\left(\begin{array}{c}
\xi \\
\xi_0 \\
\mu
\end{array}\right)+\Theta(t) \int_0^t \Theta^{-1}(s) \Delta d s
$$
where $\Theta(\cdot)$ is the fundamental solution matrix of the ODE: $\dot{\Psi}=\Pi \Psi
$ with the initial conditions $\Psi(0)=I_3$ and $\mathbb{E}\varphi$ need to satisfy the  initial conditions $\mathbb{E}\varphi(0)=\mu$. Noting the terminal condition in (\ref{dimensional expansion of FBODE}), now we present the following result.

\begin{assumption}\label{Assumption FBODE}
	$$
	det\left[\left(\begin{array}{ccc}
		0 & 0 & 1 \\
	\end{array}\right) \Theta(T)\left(\begin{array}{l}
		0 \\
		0  \\
		1
	\end{array}\right)\right]\neq 0,
	$$
\end{assumption}
\begin{mypro}
For given $T>0$, let Assumption \ref{Assumption FBODE} holds, then (\ref{FBODE}) has a unique solution on $[0, T]$, for any initial value $\xi$ and $\xi_0$.
\end{mypro}

\begin{remark}
Similar results about the solvability of coupled FBODE can be referred to the Chapter 2, Section 3 of Ma and Yong \cite{Ma-Yong-1999} and Section 4 of Huang and Huang \cite{Huang-Huang-2017}.
\end{remark}

Then, the state feedback representation of the decentralized optimal strategies of the followers can be obtained in the following theorem.
\begin{thm}\label{theorem feedback of follower}
Let Assumptions \ref{A2.1}, \ref{A2.2}, \ref{A3.1} and \ref{Assumption FBODE} hold, for given $u_0(\cdot) \in \mathcal{U}_{0}^{l,d}$, the state feedback optimal strategies of the followers $\mathcal{A}_i,i=1,2,\cdots$, can be represented as
\begin{equation}\label{follower new feedback strategy}
\begin{aligned}
u_i^* = &-\mathcal{R}_1^{-1}\left[ \left(B_1 + D_1 C_1 + \widetilde{D}_1 \widetilde{C}_1\right) P_1 \hat{\bar{x}}_i^* + \left(B_1 P_2 + D_1 F_1 P_1 + \widetilde{D}_1 \widetilde{F}_1 P_1\right) \mathbb{E}z \right. \\
& \left. + B_1 \mathbb{E} \varphi + D_1 P_1 \sigma_1 + \widetilde{D}_1 P_1 \widetilde{\sigma}_1 \right],\ \mathbb{P}\text{-a.s.},\ \text{a.e.}\ t\in[0,T],
\end{aligned}
\end{equation}
where $\hat{\bar{x}}_i^*(\cdot)\in L^2_{\mathcal{F}_t^{W_i}}(0,T;\mathbb{R})$ satisfies
\begin{equation}\label{hat bar x_i}
\left\{\begin{aligned}
d \hat{\bar{x}}_i^* =& \left\{\left[ A_1-B_1\mathcal{R}_1^{-1}\left(B_1 + D_1 C_1 + \widetilde{D}_1 \widetilde{C}_1\right) P_1\right] \hat{\bar{x}}_i^*\right.  \\
&\ +\left[E_1-B_1\mathcal{R}_1^{-1} \left(B_1 P_2 + D_1 F_1 P_1 + \widetilde{D}_1 \widetilde{F}_1 P_1\right)\right] \mathbb{E}z-B_1\mathcal{R}_1^{-1}B_1 \mathbb{E} \varphi\\
&\ \left. -B_1\mathcal{R}_1^{-1}\left( D_1 P_1 \sigma_1 + \widetilde{D}_1 P_1 \widetilde{\sigma}_1\right) +b_1\right\} d t \\
& +\left\{\left[ C_1-D_1\mathcal{R}_1^{-1}\left(B_1 + D_1 C_1 + \widetilde{D}_1 \widetilde{C}_1\right) P_1\right] \hat{\bar{x}}_i^*\right.  \\
&\ +\left[F_1-D_1\mathcal{R}_1^{-1} \left(B_1 P_2 + D_1 F_1 P_1 + \widetilde{D}_1 \widetilde{F}_1 P_1\right)\right] \mathbb{E}z-D_1\mathcal{R}_1^{-1}B_1 \mathbb{E} \varphi\\
&\ \left. -D_1\mathcal{R}_1^{-1}\left( D_1 P_1 \sigma_1 + \widetilde{D}_1 P_1 \widetilde{\sigma}_1\right) +\sigma_1\right\}d W_i, \\
 \hat{\bar{x}}_i^* (0) =&\ \xi,
\end{aligned}\right.
\end{equation}
and $\mathbb{R}$-valued function triple $(\mathbb{E}z(\cdot),\mathbb{E}\bar{x}_0(\cdot),\mathbb{E} \varphi)$ satisfies (\ref{FBODE}).
\end{thm}

\begin{proof}
(\ref{follower new feedback strategy}) can be obtained from (\ref{follower feedback strategy}) since $\hat{\xi}(\cdot)\equiv0$ in (\ref{equation of hat varphi}). (\ref{hat bar x_i}) can be obtained by putting (\ref{follower new feedback strategy}) into (\ref{follower limiting state}) and taking the conditional expectation $\mathbb{E}[\cdot \mid \mathcal{F}_t^{W_i}]$ on both sides of it.  The proof is complete.
\end{proof}

For notational simplicity and noticing (\ref{equation of E varphi}), we represent the state feedback optimal strategies (\ref{follower new feedback strategy}) of the followers as
\begin{equation}\label{follower feedback form simple}
u_i^*=K_1 \hat{\bar{x}}_i^*+K_2 \mathbb{E}z+K_3-\mathcal{R}_1^{-1} B_1 \mathbb{E} \varphi,\ \mathbb{P}\text{-a.s.},\ \text{a.e.}\ t\in[0,T],
\end{equation}
where
$$
\begin{aligned}
& \left\{\begin{array}{l}
K_1:=\mathcal{R}_1^{-1}\left(B_1+D_1 C_1+\widetilde{D}_1 \widetilde{C}_1\right)P_1 ,\quad K_2:=-\mathcal{R}_1^{-1}\left(B_1 P_2+D_1 F_1 P_1+\widetilde{D}_1 \widetilde{F}_1 P_1\right) ,\\
K_3:=-\mathcal{R}_1^{-1}\left(D_1 P_1 \sigma_1+\widetilde{D}_1 P_1 \widetilde{\sigma}_1\right),
\end{array}\right.
\end{aligned}
$$
and then
\begin{equation}\label{follower  E feedback form simple}
\mathbb{E} u_i^*=\left(K_1+K_2\right) \mathbb{E} z+K_3-\mathcal{R}_1^{-1} B_1 \mathbb{E} \varphi,\quad\mbox{for }i=1,2,\cdots.
\end{equation}

To conclude this subsection, let's discuss the well-posedness of the CC system (\ref{CC system}). Similar as Li et al. {\cite{Li-Nie-Wu-2023}}, we essentially employed the Riccati method.

After obtaining the optimal strategies of the followers in their feedback form (\ref{follower  E feedback form simple}), we substitute them into (\ref{CC system}) to obtain
\begin{equation}\label{CC system of decouple}
\left\{\begin{aligned}
 d \bar{x}_i^*=&\ \Big\{A_1 \bar{x}_i^*+B_1\left(K_1 \hat{\bar{x}}_i^*+K_2 \mathbb{E}z+K_3-\mathcal{R}_1^{-1} B_1 \mathbb{E} \varphi\right)+E_1 \mathbb{E}\left[\bar{x}_i^* \mid \mathcal{F}_t^{W}\right]
+b_1\Big\} d t \\
& +\Big\{C_1 \bar{x}_i^*+D_1\left(K_1 \hat{\bar{x}}_i^*+K_2 \mathbb{E}z+K_3-\mathcal{R}_1^{-1} B_1 \mathbb{E} \varphi\right)+F_1 \mathbb{E}\left[\bar{x}_i^* \mid \mathcal{F}_t^{W}\right]
+\sigma_1\Big\} d W_i \\
& +\left\{\widetilde{C}_1 \bar{x}_i^*+D_1\left(K_1 \hat{\bar{x}}_i^*+K_2 \mathbb{E}z+K_3-\mathcal{R}_1^{-1} B_1 \mathbb{E} \varphi\right)+\widetilde{F}_1 \mathbb{E}\left[\bar{x}_i^* \mid \mathcal{F}_t^{W}\right]
+\widetilde{\sigma}_1\right\} dW, \\
d \bar{x}_0 =&\ \Big[A_0 \bar{x}_0+B_0 u_0+E_0 \mathbb{E}\left[\bar{x}_i^* \mid \mathcal{F}_t^{W}\right]+b_0\Big] d t \\
& +\Big[C_0 \bar{x}_0+D_0 u_0+F_0 \mathbb{E}\left[\bar{x}_i^* \mid \mathcal{F}_t^{W}\right]+\sigma_0\Big] d W_0 \\
& +\left[\widetilde{C}_0 \bar{x}_0+\widetilde{D}_0 u_0+\widetilde{F}_0 \mathbb{E}\left[\bar{x}_i^* \mid \mathcal{F}_t^{W}\right]+\widetilde{\sigma}_0\right] d W , \\
d p_i=&-\left\{A_1 p_i+C_1 q_i+\widetilde{C}_1 \widetilde{q}_i-Q_1\left[\bar{x}_i^*-\left(\lambda \bar{x}_0+\left(1-\lambda\right) \mathbb{E}\left[\bar{x}_i^* \mid \mathcal{F}_t^{W}\right]
\right)\right]\right\} d t\\
&+q_i d W_i+\widetilde{q}_i d W+q_{i,0} d W_0 ,\\
 \bar{x}_i^*(0)=&\ \xi,\quad\bar{x}_0(0)=\xi_0,\quad p_i(T)=-G_1\bar{x}_i^*(T),\quad\mbox{for }i=1,2,\cdots.
\end{aligned}\right.
\end{equation}
Since we have already established the solvability of the FBODE (\ref{FBODE}), the first equation of (\ref{CC system of decouple}) is a filtered SDE. Taking $\mathbb{E}[\cdot \mid \mathcal{F}_t^{W}]$ on both sides of it, we get
\begin{equation}\label{equation of E[x|F^W]}
\left\{\begin{aligned}
 d \mathbb{E}\left[\bar{x}_i^* \mid \mathcal{F}_t^{W}\right]=
 &\ \Big\{(A_1+E_1) \mathbb{E}\left[\bar{x}_i^* \mid \mathcal{F}_t^{W}\right]+B_1\left(K_1 \mathbb{E}\bar{x}_i^*+K_2 \mathbb{E}z+K_3-\mathcal{R}_1^{-1} B_1 \mathbb{E} \varphi\right) +b_1\Big\} d t \\
& +\left\{(\widetilde{C}_1 + \widetilde{F}_1) \mathbb{E}\left[\bar{x}_i^* \mid \mathcal{F}_t^{W}\right]+K_1 \mathbb{E}\bar{x}_i^*+K_2 \mathbb{E}z+K_3-\mathcal{R}_1^{-1} B_1 \mathbb{E} \varphi +\widetilde{\sigma}_1\right\} dW, \\
\mathbb{E}\left[\bar{x}_i^* \mid \mathcal{F}_t^{W}\right](0)=&\ \xi,\quad\mbox{for }i=1,2,\cdots.
\end{aligned}\right.
\end{equation}
This is a linear SDE, we can easily derive its solvability by standard SDE theory. And the solvability of (\ref{hat bar x_i}) of $\hat{\bar{x}}_i(\cdot)$ can be guaranteed similarly. After these, we find that the first equation of (\ref{CC system of decouple}) is a linear SDE, which is also solvable. And the second equation of (\ref{CC system of decouple}) is a linear SDE, which also admits a unique solution $\bar{x}_0(\cdot)$. Then, the linear BSDE in (\ref{CC system of decouple}) admits a unique solution $\left(p_i(\cdot), q_i(\cdot), \widetilde{q}_i(\cdot), q_{i,0}(\cdot)\right)$. Therefore, the well-posedness of (\ref{CC system}) is obtained.

\subsection{Open loop decentralized strategy of the leader}

Noting that every follower take optimal strategies $u_i^*(\cdot)$, the leader $\mathcal{A}_0$ will face the following ``new" centralized state equation
\begin{equation}\label{leader state centralized new }
\left\{\begin{aligned}
d x_0= & \left[A_0 x_0+B_0 u_0+E_0 x^{*(N)}+b_0\right] d t  +\left[C_0 x_0+D_0 u_0+F_0 x^{*(N)}+\sigma_0\right] d W_0 \\
& +\left[\widetilde{C}_0 x_0+\widetilde{D}_0 u_0+\widetilde{F}_0 x^{*(N)}+\widetilde{\sigma}_0\right] d W , \\
d x_i^* =& \left[A_1 x_i+B_1\left(K_1 \hat{\bar{x}}_i^*+K_2 \mathbb{E}\bar{x}_i^*+K_3-\mathcal{R}_1^{-1} B_1 \mathbb{E} \varphi\right)+E_1 x^{*(N)}+b_1\right] d t \\
& +\left[C_1 x_i+D_1\left(K_1 \hat{\bar{x}}_i^*+K_2 \mathbb{E}\bar{x}_i^*+K_3-\mathcal{R}_1^{-1} B_1 \mathbb{E} \varphi\right)+F_1 x^{*(N)}+\sigma_1\right] d W_i \\
& +\left[\widetilde{C}_1 x_i+\widetilde{D}_1\left(K_1 \hat{\bar{x}}_i^*+K_2 \mathbb{E}\bar{x}_i^*+K_3-\mathcal{R}_1^{-1} B_1 \mathbb{E} \varphi\right)+\widetilde{F}_1 x^{*(N)}+\widetilde{\sigma}_1\right]  d W ,\\
d \mathbb{E} \varphi=&\left(M_1 \mathbb{E} \varphi+Q_1\lambda \mathbb{E} \bar{x}_0+M_2\right) d t, \\
 x_0(0)=&\ \xi_0, \quad x_i^*(0) =\xi, \quad \mathbb{E} \varphi(T)=0,\quad\mbox{for }i=1,2,\cdots,
\end{aligned}\right.
\end{equation}
where $x^{*(N)}(\cdot) := \frac{1}{N} \sum_{i=1}^N x_i^*(\cdot)$.

The centralized cost functional (\ref{leader cost}) of the leader $\mathcal{A}_0$ now writes
\begin{equation}\label{leader centralized cost functional new }
\mathcal{J}_0\left(u_0(\cdot)\right)=\frac{1}{2} \mathbb{E} \int_0^T\left[R_0 u_0^2+Q_0\left(x_0-x^{*(N)}\right)^2\right] d t+G_0 x_0^2(T).
\end{equation}
Then, from (\ref{leader limiting state}), (\ref{SLLN}), (\ref{z state}) and (\ref{follower feedback form simple}), the corresponding decentralized state equation of the leader $\mathcal{A}_0$ is
\begin{equation}\label{leader state decentralized new }
\left\{\begin{aligned}
d \bar{x}_0= & \Big[A_0 \bar{x}_0+B_0 u_0+E_0 z+b_0\Big] d t +\Big[C_0 \bar{x}_0+D_0 u_0+F_0 z+\sigma_0\Big] d W_0 \\
& +\Big[\widetilde{C}_0 \bar{x}_0+\widetilde{D}_0 u_0+\widetilde{F}_0 z+\widetilde{\sigma}_0\Big] d W,\\
 d z=&\Big[\left(A_1+E_1\right) z+B_1\left[\left(K_1+K_2\right) \mathbb{E} z+K_3-\mathcal{R}_1^{-1} B_1 \mathbb{E} \varphi\right]+b_1\Big] d t \\
&+\left[\left(\widetilde{C}_1+\widetilde{F}_1\right) z+\widetilde{D}_1\left[\left(K_1+K_2\right) \mathbb{E} z+K_3-\mathcal{R}_1^{-1} B_1 \mathbb{E} \varphi\right]+\widetilde{\sigma}_1\right] d W, \\
d \mathbb{E} \varphi=&\left(M_1 \mathbb{E} \varphi+Q_1\lambda \mathbb{E} \bar{x}_0+M_2\right) d t, \\
\bar{x}_0(0) =&\ \xi_0, \quad z(0)=\xi, \quad \mathbb{E} \varphi(T)=0,
\end{aligned}\right.
\end{equation}
and the centralized cost functional is
\begin{equation}\label{leader decentralized cost functional new}
J_0\left(u_0(\cdot)\right)=\frac{1}{2} \mathbb{E} \int_0^T\left[R_0 u_0^2+Q\left(\bar{x}_0-z\right)^2\right] d t+G \bar{x}_0^2(T).
\end{equation}

By applying the stochastic maximum principle with partial information, we can obtain the open-loop decentralized strategies of the leader $\mathcal{A}_0$ for Problem \ref{problem decentralized} (ii). For simplicity of notation, we let $\check{\xi}:=\mathbb{E}\left[\left.\xi\right\rvert \mathcal{F}_t^{W_0}\right]$ and the proof can be seen in this paper's arXiv version \cite{Si-Shi-arXiv2024}.

\begin{thm}\label{open loop thm1 of leader}
Let Assumptions \ref{A2.1}, \ref{A2.2}, \ref{A3.1} and \ref{Assumption FBODE} hold. Then the open-loop decentralized strategy of the leader $\mathcal{A}_0$ is given by
\begin{equation}\label{leader open loop optimal strategy}
u_0^*=R_0^{-1}\left(B_0 \check{y}_0+D_0 \check{z}_0+\widetilde{D}_0 \check{\widetilde{z}}_0\right),\ \mathbb{P}\text{-a.s.},\ \text{a.e.},
\end{equation}
where $\left(\bar{x}_0^*(\cdot),z^*(\cdot),\mathbb{E} \varphi^*(\cdot),y_0(\cdot),z_0(\cdot),\widetilde{z}_0(\cdot),g(\cdot),\widetilde{h}(\cdot),h_0(\cdot),\eta(\cdot)\right)$ satisfies the following stochastic Hamiltonian system
\end{thm}
\begin{equation}\label{leader Hamiltonian system}
\left\{\begin{aligned}
d \bar{x}_0^*= & {\left[A_0 \bar{x}_0^*+B_0 R_0^{-1}\left(B_0 \check{y}_0+D_0 \check{z}_0+\widetilde{D}_0 \check{\widetilde{z}}_0\right)+E_0 z^*+b_0\right] d t }\\
&+\left[C_0 \bar{x}_0^*+D_0 R_0^{-1}\left(B_0 \check{y}_0+D_0 \check{z}_0+\widetilde{D}_0 \check{\widetilde{z}}_0\right)+F_0 z^*+\sigma_0\right] d W_0 \\
& +\left[\widetilde{C}_0\bar{x}_0^*+\widetilde{D}_0 R_0^{-1}\left(B_0 \check{y}_0+D_0 \check{z}_0+\widetilde{D}_0 \check{\widetilde{z}}_0\right)+\widetilde{F}_0 z^*+\widetilde{\sigma}_0\right] d W,\\
d z^*=&\Big[\left(A_1+E_1\right) z^*+B_1\left[\left(K_1+K_2\right) \mathbb{E} z^*+K_3-\mathcal{R}_1^{-1} B_1 \mathbb{E} \varphi^*\right]+b_1\Big] d t, \\
&+\left[\left(\widetilde{C}_1+\widetilde{F}_1\right) z^*+\widetilde{D}_1\left[\left(K_1+K_2\right) \mathbb{E} z^*+K_3-\mathcal{R}_1^{-1} B_1 \mathbb{E} \varphi^*\right]+\widetilde{\sigma}_1\right] d W ,\\
d \mathbb{E} \varphi^*=&\left(M_1 \mathbb{E} \varphi^*+Q_1 \lambda \mathbb{E} \bar{x}_0^*+M_2\right) d t, \\
d y_0 =&-\left(A_0 y_0+C_0 z_0+\widetilde{C}_0 \widetilde{z}_0-Q_0\left(\bar{x}_0^*-z^*\right)+Q_1 \lambda \mathbb{E} \eta\right) d t+z_0 d W_0+\widetilde{z}_0 d W,\\
d g= & -\left[E_0 y_0+F_0 z_0+\widetilde{F}_0 \widetilde{z}_0+\left(A_1+E_1\right) g+\left(\widetilde{C}_1+\widetilde{F}_1\right) \widetilde{h}+Q_0\left(\bar{x}_0^*-z^*\right)\right. \\
& \left.+B_1\left(K_1+K_2\right) \mathbb{E} g+\widetilde{D}_1\left(K_1+K_2\right) \mathbb{E} h\right]d t+\widetilde{h} d W+h_0 d W_0 ,\\
d \eta =&-\left(-\mathcal{R}_1^{-1} B_1^2 \mathbb{E} g-\mathcal{R}_1^{-1} B_1 \widetilde{D}_1 \mathbb{E} \widetilde{h}+M_1 \mathbb{E} \eta\right) d t,\\
\bar{x}_i^*(0) =&\ \xi_0,\ z^*(0)=\xi,\ \mathbb{E} \varphi^*(T)=0,\ y_0(T)  =-G_0\bar{x}_0^*(T),\ g(T)=0,\ \eta(0)=0.
\end{aligned}\right.
\end{equation}

Noticing (\ref{leader Hamiltonian system}) is also a CMF-FBSDE, we still discuss its well-posedness in the following subsection.

\subsection{Feedback decentralized strategy of the leader}

Now, we derive the state feedback of the open loop strategy of the leader $\mathcal{A}_0$, by the dimension expansion technique of Yong \cite{Yong-2002}. Let
$$
X:=\left[\begin{array}{c}
	\bar{x}_0^* \\
	z^* \\
	\eta
\end{array}\right],\quad Y:=\left[\begin{array}{c}
	y_0 \\
	g \\
	\mathbb{E} \varphi^*
\end{array}\right],\quad Z_0:=\left[\begin{array}{c}
	z_0 \\
	h_0 \\
	0
\end{array}\right],\quad \widetilde{Z}:=\left[\begin{array}{c}
	\widetilde{z}_0 \\
	\widetilde{h} \\
	0
\end{array}\right].
$$
Then, the Hamiltonian system (\ref{leader Hamiltonian system}) of the leader $\mathcal{A}_0$ can be rewritten as
\begin{equation}\label{dimension expansion}
	\left\{\begin{aligned}
		d X= & \left(L_{11} X+L_{12} \mathbb{E} X+L_{13} \mathbb{E} Y+L_{14} \mathbb{E} \widetilde{Z}+\mathcal{B}_0 u_0^*+f_1\right) d t \\
		& +\left(L_{21} X+\mathcal{D}_0 u_0^*+f_2\right) d W_0  +\left(L_{31} X+L_{32} \mathbb{E} X-L_{14}^\top \mathbb{E} Y+\widetilde{\mathcal{D}}_0 u_0^*+f_3\right) d W ,\\
		d Y= & \left(N_{11} X+N_{12} \mathbb{E} X-L_{11}^\top Y-L_{12}^\top \mathbb{E} Y-L_{21}^\top Z_0-L_{31}^\top \widetilde{Z}-L_{32}^\top \mathbb{E} \widetilde{Z}+f_4\right) d t \\
		& +Z_0 d W_0+\widetilde{Z} d W ,\\
		X(0)= &\ \Xi,\quad Y(T)=-\mathcal{G}_0 X(T),
	\end{aligned}\right.
\end{equation}
where we have denoted
$$
\begin{aligned}
	& L_{11}:=\left[\begin{array}{ccc}
		A_0 & E_0 & 0 \\
		0 & A_1+E_1 & 0 \\
		0 & 0 & -M_1
	\end{array}\right],
	L_{12}:=\left[\begin{array}{ccc}
		0 & 0 & 0 \\
		0 & B_1 (K_1+K_2) & 0 \\
		0 & 0 & 0
	\end{array}\right],
	\mathcal{B}_0:=\left[\begin{array}{c}
		B_0  \\
		0 \\
		0
	\end{array}\right],  \\
	& L_{13}:=\left[\begin{array}{ccc}
		0 & 0 & 0 \\
		0 & 0 & -\mathcal{R}_1^{-1} B_1^2 \\
		0 & \mathcal{R}_1^{-1} B_1^2 & 0
	\end{array}\right],
	L_{14}:=\left[\begin{array}{ccc}
		0 & 0 & 0 \\
		0 & 0 & 0 \\
		0 & \mathcal{R}_1^{-1} B_1 \widetilde{D}_1 & 0
	\end{array}\right] ,
	f_1:=\left[\begin{array}{c}
		b_0 \\
		B_0 K_3+b_1 \\
		0
	\end{array}\right],\\
	& L_{21}:=\left[\begin{array}{ccc}
		C_0 & F_0 & 0 \\
		0 & 0 & 0 \\
		0 & 0 & 0
	\end{array}\right] ,
	\mathcal{D}_0:=\left[\begin{array}{c}
		D_0  \\
		0  \\
		0
	\end{array}\right] ,
	f_2:=\left[\begin{array}{c}
		\sigma_0 \\
		0 \\
		0
	\end{array}\right],
	L_{31}:=\left[\begin{array}{ccc}
		\widetilde{C}_0 & \widetilde{F}_0 & 0 \\
		0 & \widetilde{C}_1+\widetilde{F}_1 & 0 \\
		0 & 0 & 0
	\end{array}\right] ,\\
	& L_{32}:=\left[\begin{array}{ccc}
		0 & 0 & 0 \\
		0 & \widetilde{D}_1\left(K_1+K_2\right) & 0 \\
		0 & 0 & 0
	\end{array}\right] ,
	\widetilde{\mathcal{D}}_0:=\left[\begin{array}{c}
		\widetilde{D}_0 \\
		0  \\
		0
	\end{array}\right] ,
	f_3:=\left[\begin{array}{c}
		\widetilde{\sigma}_0 \\
		\widetilde{D}_0 K_3+\tilde{\sigma}_1 \\
		0
	\end{array}\right],\\
	& N_{11}:=\left[\begin{array}{ccc}
		Q_0 & -Q_0 & 0 \\
		-Q_0 & Q_0 & 0 \\
		0 & 0 & 0
	\end{array}\right] ,
	N_{12}:=\left[\begin{array}{ccc}
		0 & 0 & -Q_1 \lambda \\
		0 & 0 & 0 \\
		Q_1 \lambda & 0 & 0
	\end{array}\right] ,\\
	& f_{4}:=\left[\begin{array}{l}
		0 \\
		0 \\
		M_2
	\end{array}\right],
	\Xi:=\left[\begin{array}{l}
		\xi_0 \\
		\xi \\
		0
	\end{array}\right],
	\mathcal{G}_0:=\left[\begin{array}{ccc}
		G_0 & 0 & 0 \\
		0 & 0 & 0 \\
		0 & 0 & 0
	\end{array}\right] .
\end{aligned}
$$
Noticing the terminal condition of (\ref{dimension expansion}), we can assume that
\begin{equation}\label{leader decouple equation}
Y(t)=-\Gamma_1(t) X(t)-\Gamma_2(t) \mathbb{E} X(t)-\Phi(t),\quad t\in[0,T],
\end{equation}
where $\Gamma_1(\cdot)$ and $\Gamma_2(\cdot)$ are deterministic matrix-valued functions satisfying $\Gamma_1(T)=\mathcal{G}_0$ and $\Gamma_2(T)=\mathrm{O}_{3 \times 3}$, and $(\Phi(\cdot),\psi_0(\cdot))$ are an $\mathcal{F}_t^{W_0}$-adapted process pair of the solution to the following BSDE:
\begin{equation}\label{Phi equation}
	\left\{\begin{array}{l}
		d \Phi=\gamma d t+\psi_0 d W_0 ,\\
		\Phi(T)=\mathrm{O}_{3 \times 3}.
	\end{array}\right.
\end{equation}
for some $\mathcal{F}_t^{W_0}$-adapted process $\gamma(\cdot)$ to be determined later. In addition, from the first equation of (\ref{dimension expansion}), we have
\begin{equation}\label{EX equation}
	\left\{\begin{aligned}
		d \mathbb{E} X&=\left[\left(L_{11}+L_{12}\right) \mathbb{E} X+L_{13} \mathbb{E} Y+L_{14} \mathbb{E} \widetilde{Z}+\mathcal{B}_0 \mathbb{E} u_0^*+f_1\right] d t, \\
		\mathbb{E} X(0)&=\ \Xi.
	\end{aligned}\right.
\end{equation}
Applying It\^{o}'s formula, we have
$$
\begin{aligned}
	d Y= & -\Gamma_1 d X-\dot{\Gamma}_1 X d t-\dot{\Gamma}_2 \mathbb{E} X d t-\Gamma_2 d \mathbb{E} X-d \Phi \\
	= & -\Gamma_1\left\{\left(L_{11} X+L_{12} \mathbb{E} X+L_{13} \mathbb{E} Y+L_{14} \mathbb{E} \widetilde{Z}+\mathcal{B}_0 u_0^*+f_1\right) d t\right. \\
	& +\left(L_{21} X+\mathcal{D}_0 u_0^*+f_2\right) d W_0 \left.+\left(L_{31} X+L_{32} \mathbb{E} X-L_{14}^\top \mathbb{E}Y+\widetilde{\mathcal{D}}_0 u_0^*+f_3\right) d W\right\} \\
	&-  \dot{\Gamma}_1 X d t-\dot{\Gamma}_2 \mathbb{E} X d t-\Gamma_2\left[\left(L_{11}+L_{12}\right) \mathbb{E} X+L_{13} \mathbb{E}Y+L_{14} \mathbb{E} \widetilde{Z}+\mathcal{B}_0 \mathbb{E} u_0^*+f_1\right] d t \\
	& -\gamma d t-\psi_0 d W_0 \\
	= & \left(N_{11} X+N_{12} \mathbb{E} X-L_{11}^\top Y-L_{12}^\top \mathbb{E} Y-L_{21}^\top Z_0-L_{31}^\top \widetilde{Z}-L_{32}^\top \mathbb{E} \widetilde{Z}+f_4\right) d t \\
	& +Z_0 d W_0+\widetilde{Z} d W.
\end{aligned}
$$
By comparing the coefficients of the drift terms and the diffusion terms, we obtain
\begin{equation}\label{compare the coefficients of leader }
\left\{\begin{aligned}
	 -\gamma=&\left(\dot{\Gamma}_1+\Gamma_1 L_{11}+N_{11}\right) X+\left[\dot{\Gamma}_2+\Gamma_1 L_{12}+\Gamma_2\left(L_{11}+L_{12}\right)+N_{12}\right] \mathbb{E} X-L_{11}^\top Y \\
	& +\left[\Gamma_1 L_{13}+\Gamma_2 L_{13}-L_{12}^\top\right] \mathbb{E} Y-L_{21}^\top Z_0-L_{31}^\top \widetilde{Z}+\left(\Gamma_1 L_{14}+\Gamma_2 L_{14}-L_{32}^\top\right) \mathbb{E} \widetilde{Z} \\
	& +\Gamma_1 \mathcal{B}_0 u_0^*+\Gamma_2 \mathcal{B}_0 \mathbb{E} u_0^*+\Gamma_1 f_1+\Gamma_2 f_1+f_4 ,\\
	   Z_0=&-\Gamma_1 L_{21} X-\Gamma_1 \mathcal{D}_0 u_0^*-\Gamma_1 f_2-\psi_0 ,\ \mathbb{P}\text{-a.s.},\\
\widetilde{Z}=&-\Gamma_1 L_{31} X-\Gamma_1 L_{32} \mathbb{E} X+\Gamma_1 L_{14}^{\top} \mathbb{E} Y-\Gamma_1 \widetilde{\mathcal{D}}_0 u_0^*-\Gamma_1 f_3,\ \mathbb{P}\text{-a.s.}.
\end{aligned}\right.
\end{equation}
Taking $\mathbb{E}\left[\left.\cdot\right\rvert \mathcal{F}_t^{W_0}\right]$ on the both sides of (\ref{leader decouple equation}) and the last two equalities in (\ref{compare the coefficients of leader }), we have
\begin{equation}\label{check Y Z}
\left\{\begin{aligned}
	\check{Y}&=-\Gamma_1 \check{X}-\Gamma_2 \mathbb{E} X-\check{\Phi} ,\\
	\check{Z}_0&=-\Gamma_1 L_{21} \check{X}-\Gamma_1 \mathcal{D}_0 u_0^*-\Gamma_1 f_2-\check{\psi}_0,\ \mathbb{P}\text{-a.s.}, \\
	\check{\widetilde{Z}}&=-\Gamma_1 L_{31} \check{X}-\Gamma_1 L_{32} \mathbb{E} X+\Gamma_1 L_{14}^\top\mathbb{E} Y-\Gamma_1 \widetilde{\mathcal{D}_0} u_0^*-\Gamma_1 f_3.\\
	&=-\Gamma_1 L_{31}\check{X}-\Gamma_1 (L_{32}+L_{14}^\top\Gamma_1+L_{14}^\top\Gamma_2)\mathbb{E} X-\Gamma_1 L_{14}^\top \mathbb{E} \Phi-\Gamma_1 \widetilde{\mathcal{D}}_0 u^*_0-\Gamma_1 f_3,\ \mathbb{P}\text{-a.s.}.
\end{aligned}\right.
\end{equation}
Therefore, the open loop decentralized strategy (\ref{leader open loop optimal strategy}) of the leader $\mathcal{A}_0$ can be rewritten as
$$
\begin{aligned}
	& \left(R_0+\mathcal{D}_0^\top \Gamma_1 \mathcal{D}_0+\widetilde{\mathcal{D}}_0^\top \Gamma_1 \widetilde{\mathcal{D}}_0\right) u_0^* \\
	=&-\left[\left(\mathcal{B}_0^\top \Gamma_1+\mathcal{D}_0^\top \Gamma_1 L_{21}+\widetilde{\mathcal{D}}^\top_0 \Gamma_1 L_{31}\right) \check{X}
    +\left(\mathcal{B}_0^\top \Gamma_2+\widetilde{\mathcal{D}}_0^\top \Gamma_1 L_{32}+\widetilde{\mathcal{D}}_0^\top \Gamma_1 L_{14}^\top (\Gamma_1 +\Gamma_2)\right) \mathbb{E} X\right. \\
	&\quad \left.+\mathcal{B}_0^\top \check{\Phi}+\mathcal{D}_0^\top \Gamma_1 f_2+\mathcal{D}_0^\top \check{\psi}_0+\widetilde{\mathcal{D}}_0^\top \Gamma_1 L_{14}^\top \mathbb{E} \Phi
    +\widetilde{\mathcal{D}}_0^\top\Gamma_1 f_3\right], \ \mathbb{P}\text{-a.s.},\ \text{a.e.}\ t\in[0,T].
\end{aligned}
$$

Now, we need to assume that
\begin{assumption}\label{A3.2}\quad
$\mathcal{R}_0:=R_0+\mathcal{D}_0^\top \Gamma_1 \mathcal{D}_0+\widetilde{\mathcal{D}}_0^\top \Gamma_1 \widetilde{\mathcal{D}}_0 \neq 0,$
\end{assumption}
Then
\begin{equation}\label{leader feedback strategy}
\begin{aligned}
	u_0^*=&-\mathcal{R}_0^{-1}\left[\left(\mathcal{B}_0^\top \Gamma_1+\mathcal{D}_0^\top \Gamma_1 L_{21}+\widetilde{\mathcal{D}}^\top_0 \Gamma_1 L_{31}\right) \check{X}\right.\\
    &\quad +\left(\mathcal{B}_0^\top \Gamma_2+\widetilde{\mathcal{D}}_0^\top \Gamma_1 L_{32}+\widetilde{\mathcal{D}}_0^\top \Gamma_1 L_{14}^\top (\Gamma_1+\Gamma_2)\right) \mathbb{E} X\\
    &\quad \left.+\mathcal{B}_0^\top\check{\Phi}+\widetilde{\mathcal{D}}_0^\top \Gamma_1 L_{14}^\top \mathbb{E}\Phi+\mathcal{D}_0^\top \Gamma_1 f_2+\mathcal{D}_0^\top \check{\psi}_0
    +\widetilde{\mathcal{D}}_0^\top\Gamma_1 f_3\right], \ \mathbb{P}\text{-a.s.},\ \text{a.e.}\ t\in[0,T],
\end{aligned}
\end{equation}
and
\begin{equation}\label{leader E feedback strategy}
\begin{aligned}
	\mathbb{E} u_0^*=&-\mathcal{R}_0^{-1}\left[\left(\mathcal{B}_0^\top \Gamma_1+\mathcal{D}_0^\top \Gamma_1 L_{21}+\widetilde{\mathcal{D}}^\top_0 \Gamma_1 L_{31} +\mathcal{B}_0^\top \Gamma_2
    +\widetilde{\mathcal{D}}_0^\top \Gamma_1 L_{32}+\widetilde{\mathcal{D}}_0^\top \Gamma_1 L_{14}^\top (\Gamma_1+\Gamma_2)\right) \mathbb{E} X\right. \\
	& \left.+\left(\mathcal{B}_0^\top+\widetilde{\mathcal{D}}_0^\top \Gamma_1 L_{14}^\top\right) \mathbb{E} \Phi+\mathcal{D}_0^\top \Gamma_1 f_2+\mathcal{D}_0^\top \mathbb{E}\psi_0
    +\widetilde{\mathcal{D}}_0^\top\Gamma_1 f_3\right].
\end{aligned}
\end{equation}
Similarly, taking $\mathbb{E}\left[\left.\cdot\right\rvert \mathcal{F}_t^{W_0}\right]$ on the both sides of the first equality in (\ref{compare the coefficients of leader }), and considering (\ref{leader feedback strategy}) and (\ref{leader E feedback strategy}), we get

$$
\begin{aligned}
-\check{\gamma}=&\left\{\left[\dot{\Gamma}_1+\Gamma_1 L_{11}+N_{11}+L_{11}^\top \Gamma_1+L_{21}^\top \Gamma_1 L_{21}+L_{31}^\top \Gamma_1 L_{31}\right. \right.\\
& \left.+\left(L_{21}^\top \Gamma_1 \mathcal{D}_0+L_{31}^\top \Gamma_1 \widetilde{\mathcal{D}_0}+\Gamma_1\mathcal{B}_0\right) \mathcal{R}_0^{-1}\left(-\mathcal{B}^\top_0 \Gamma_1
 -\mathcal{D}^\top_0 \Gamma_1 L_{21}-\widetilde{\mathcal{D}}^\top_0 \Gamma_1 L_{31}\right)\right] \check{X}\\
& +\left\{\dot{\Gamma}_2+\Gamma_1 L_{12}+\Gamma_2 L_{11}+\Gamma_2L_{12}+N_{12}+L_{11}^\top \Gamma_2-\Gamma_1 L_{13}\left(\Gamma_1+\Gamma_2\right)\right. \\
& -\Gamma_2 L_{13}\left(\Gamma_1+\Gamma_2\right)+L_{12}^\top\left(\Gamma_1+\Gamma_2\right)+L_{31}^\top \Gamma_1 L_{32}+L_{31}^\top \Gamma_1 L_{14}^\top \Gamma_1+L_{31}^\top \Gamma_1 L_{14}^\top \Gamma_2\\
& -\left(\Gamma_1 L_{14}+\Gamma_2 L_{14}-L_{32}^\top\right)\Gamma_1\left( L_{31}+L_{32}+L_{14}^\top \Gamma_1+L_{14}^\top \Gamma_2\right)\\
& +\left(L_{21}^\top \Gamma_1 \mathcal{D}_0+L_{31}^\top \Gamma_1 \widetilde{\mathcal{D}}_0+\Gamma_1\mathcal{B}_0\right) \mathcal{R}_0^{-1}\left(-\mathcal{B}_0^\top \Gamma_2
 -\widetilde{\mathcal{D}}_0^\top \Gamma_1 L_{32}-\widetilde{\mathcal{D}}_0^\top\Gamma_1 L_{14}^\top \Gamma_1- \widetilde{\mathcal{D}}_0^\top\Gamma_1 L_{14}^\top \Gamma_2\right) \\
&+\left[\Gamma_2\mathcal{B}_0-\left(\Gamma_1 L_{14}+\Gamma_2 L_{14}-L_{32}^\top\right) \Gamma_1 \widetilde{\mathcal{D}}_0 \right]\mathcal{R}_0^{-1}
 \left(-\mathcal{B}_0^\top \Gamma_1-\mathcal{D}_0^\top \Gamma_1 L_{21}\right.\\
&\left.\left.-\widetilde{\mathcal{D}}^\top_0 \Gamma_1 L_{31} -\mathcal{B}_0^\top \Gamma_2
 -\widetilde{\mathcal{D}}_0^\top \Gamma_1 L_{32}-\widetilde{\mathcal{D}}^\top_0 \Gamma_1 L_{14}^\top \Gamma_1-\widetilde{\mathcal{D}}_0^\top \Gamma_1 L_{14}^\top \Gamma_2\right)\right\} \mathbb{E} X\\
&+\left(L_{21}^\top \Gamma_1 \mathcal{D}_0+L_{31}^\top \Gamma_1 \widetilde{\mathcal{D}}_0+\Gamma_1\mathcal{B}_0\right) \mathcal{R}_0^{-1}\left(-\mathcal{B}_0^\top \check{\Phi}
 -\widetilde{\mathcal{D}}_0^\top \Gamma_1 L_{14}^\top\mathbb{E}\Phi-\mathcal{D}_0^\top \check{\psi}_0-\mathcal{D}_0^\top \Gamma_1 f_2\right.\\
&\left.-\widetilde{\mathcal{D}}_0^\top \Gamma_1 f_3\right)+\left[\Gamma_2\mathcal{B}_0-\left(\Gamma_1 L_{14}+\Gamma_2 L_{14}-L_{32}^\top\right) \Gamma_1 \widetilde{\mathcal{D}}_0 \right]\mathcal{R}_0^{-1}\\
&\times\left[\left(-\widetilde{\mathcal{D}}_0^\top \Gamma_1 L_{14}^\top-\mathcal{B}_0^\top\right) \mathbb{E} \Phi-\mathcal{D}_0^\top \mathbb{E} \psi_0-\mathcal{D}_0^\top \Gamma_1 f_2
 -\widetilde{\mathcal{D}}_0^\top \Gamma_1 f_3\right]+L_{11}^\top \check{\Phi}\\
&+\left[-\Gamma_1 L_{13}-\Gamma_2 L_{13}+L_{12}^\top+L_{31}^\top \Gamma_1 L_{14}^\top-\left(\Gamma_1 L_{14}+\Gamma_2 L_{14}-L_{32}^\top\right) \Gamma_1 L_{14}^\top\right] \mathbb{E} \Phi \\
& +L_{21}^\top \check{\psi}_0+L_{21}^\top \Gamma_{1} f_2+L_{31}^\top \Gamma_1 f_3-\left(\Gamma_1 L_{14}+\Gamma_2 L_{14}-L_{32}^{\top}\right) \Gamma_1 f_3+\Gamma_1 f_1+\Gamma_2 f_1+f_4 .\\
\end{aligned}
$$
Then, we can get the equation of $\Gamma_1(\cdot)$:
\begin{equation}\label{Gamma1 equation}
\left\{\begin{aligned}
    &\dot{\Gamma}_1+\Gamma_1 L_{11}+N_{11}+L_{11}^\top \Gamma_1+L_{21}^\top \Gamma_1 L_{21}+L_{31}^\top \Gamma_1 L_{31}\\
	&+\left(L_{21}^\top \Gamma_1 \mathcal{D}_0+L_{31}^\top \Gamma_1 \widetilde{\mathcal{D}_0}+\Gamma_1\mathcal{B}_0\right) \mathcal{R}_0^{-1}\left(-\mathcal{B}^\top_0 \Gamma_1
    -\mathcal{D}^\top_0 \Gamma_1 L_{21}-\widetilde{\mathcal{D}}^\top_0 \Gamma_1 L_{31}\right)=0, \\
    &\Gamma_1(T)=\mathcal{G}_0,
\end{aligned}\right.
\end{equation}
the equation of $\Gamma_2(\cdot)$:
\begin{equation}\label{Gamma2 equation}
\left\{\begin{aligned}
    &\dot{\Gamma}_2+\Gamma_1 L_{12}+\Gamma_2 L_{11}+\Gamma_2L_{12}+N_{12}+L_{11}^\top \Gamma_2-\Gamma_1 L_{13}\left(\Gamma_1+\Gamma_2\right) \\
	& -\Gamma_2 L_{13}\left(\Gamma_1+\Gamma_2\right)+L_{12}^\top\left(\Gamma_1+\Gamma_2\right)+L_{31}^\top \Gamma_1 L_{32}+L_{31}^\top \Gamma_1 L_{14}^\top \Gamma_1+L_{31}^\top \Gamma_1 L_{14}^\top \Gamma_2\\
	& -\left(\Gamma_1 L_{14}+\Gamma_2 L_{14}-L_{32}^\top\right)\Gamma_1\left( L_{31}+L_{32}+L_{14}^\top \Gamma_1+L_{14}^\top \Gamma_2\right)\\
	& +\left(L_{21}^\top \Gamma_1 \mathcal{D}_0+L_{31}^\top \Gamma_1 \widetilde{\mathcal{D}}_0+\Gamma_1\mathcal{B}_0\right) \mathcal{R}_0^{-1}\left(-\mathcal{B}_0^\top \Gamma_2
    -\widetilde{\mathcal{D}}_0^\top \Gamma_1 L_{32}-\widetilde{\mathcal{D}}_0^\top\Gamma_1 L_{14}^\top \Gamma_1\right. \\
	&\left.-\widetilde{\mathcal{D}}_0^\top\Gamma_1 L_{14}^\top \Gamma_2\right)+\left[\Gamma_2\mathcal{B}_0-\left(\Gamma_1 L_{14}+\Gamma_2 L_{14}
    -L_{32}^\top\right) \Gamma_1 \widetilde{\mathcal{D}}_0 \right]\mathcal{R}_0^{-1}\left(-\mathcal{B}_0^\top \Gamma_1\right.\\
	&\left.-\mathcal{D}_0^\top \Gamma_1 L_{21}-\widetilde{\mathcal{D}}^\top_0 \Gamma_1 L_{31} -\mathcal{B}_0^\top \Gamma_2-\widetilde{\mathcal{D}}_0^\top \Gamma_1 L_{32}
    -\widetilde{\mathcal{D}}^\top_0 \Gamma_1 L_{14}^\top \Gamma_1-\widetilde{\mathcal{D}}_0^\top \Gamma_1 L_{14}^\top \Gamma_2\right) =0,\\
    &\Gamma_2(T)=\mathrm{O}_{3 \times 3},
\end{aligned}\right.
\end{equation}
and the BSDE of $(\check{\Phi}(\cdot),\check{\psi}_0(\cdot))$:
\begin{equation}\label{equation of check Phi}
\left\{\begin{aligned}
    -d \check{\Phi}=& \left\{\left(L_{21}^\top \Gamma_1 \mathcal{D}_0+L_{31}^\top \Gamma_1 \widetilde{\mathcal{D}}_0+\Gamma_1\mathcal{B}_0\right) \mathcal{R}_0^{-1}\left[-\mathcal{B}_0^\top \check{\Phi}
    -\widetilde{\mathcal{D}}_0^\top \Gamma_1 L_{14}^\top\mathbb{E}\Phi-\mathcal{D}_0^\top \check{\psi}_0 \right.\right.\\
	&\left.-\mathcal{D}_0^\top \Gamma_1 f_2-\widetilde{\mathcal{D}}_0^\top \Gamma_1 f_3\right]+\left[\Gamma_2\mathcal{B}_0-\left(\Gamma_1 L_{14}+\Gamma_2 L_{14}
    -L_{32}^\top\right) \Gamma_1 \widetilde{\mathcal{D}}_0 \right]\mathcal{R}_0^{-1}\\
	&\times\left[\left(-\widetilde{\mathcal{D}}_0^\top \Gamma_1 L_{14}^\top-\mathcal{B}_0^\top\right) \mathbb{E} \Phi-\mathcal{D}_0^\top \mathbb{E} \psi_0-\mathcal{D}_0^\top \Gamma_1 f_2
    -\widetilde{\mathcal{D}}_0^\top \Gamma_1 f_3\right]+L_{11}^\top \check{\Phi}\\
	&+\left[-\Gamma_1 L_{13}-\Gamma_2 L_{13}+L_{12}^\top+L_{31}^\top \Gamma_1 L_{14}^\top-\left(\Gamma_1 L_{14}+\Gamma_2 L_{14}-L_{32}^\top\right) \Gamma_1 L_{14}^\top\right] \mathbb{E} \Phi \\
	&\left. +L_{21}^\top \check{\psi}_0+L_{21}^\top \Gamma_{1} f_2+L_{31}^\top \Gamma_1 f_3-\left(\Gamma_1 L_{14}+\Gamma_2 L_{14}-L_{32}^\top\right) \Gamma_1 f_3\right.\\
    &+\Gamma_1 f_1+\Gamma_2 f_1+f_4 \Big\} dt-\check{\psi}_0 d W_0,\\
    \check{\Phi}(T)=&\ \mathrm{O}_{3 \times 3}.
\end{aligned}\right.
\end{equation}
respectively. By the existence and uniqueness of the solution to the mean-field type BSDE \cite{Li-Sun-Xiong-2019} again, (\ref{equation of check Phi}) degenerates to
\begin{equation}\label{equation of E Phi}
\left\{\begin{aligned}
    -d \mathbb{E}\Phi=& \bigg\{\left\{L_{11}^\top+L_{21}^\top+\left(L_{31}+L_{32}\right)^\top\Gamma_1L_{14}^\top-(\Gamma_1+\Gamma_2)\left(L_{13}+L_{14}\Gamma_1L_{14}^\top\right.\right) \\
    &-\left[L_{21}^\top\Gamma_1\mathcal{D}_0+L_{31}^\top\Gamma_1\widetilde{\mathcal{D}}_0+\Gamma_1\mathcal{B}_0+\Gamma_2\mathcal{B}_0
    +\left(\Gamma_1L_{14}+\Gamma_2L_{14}+L_{32}^\top\right)\Gamma_1\widetilde{\mathcal{D}}_0\right]\\
	&\left.\times\mathcal{R}_0^{-1}\left(\mathcal{B}_0^\top+\widetilde{\mathcal{D}}_0^\top \Gamma_1 L_{14}^\top\right)\right\}\mathbb{E}\Phi
    +\left[\left(\Gamma_1+\Gamma_2\right)\left(L_{14}\Gamma_2\widetilde{\mathcal{D}}_0-\mathcal{B}_0\right)\right.\\
    &\left.-\left(L_{21}+L_{31}+L_{32}\right)^\top\Gamma_1\widetilde{\mathcal{D}}_0\right]\mathcal{R}_0^{-1}\left(\mathcal{D}_0^\top\Gamma_1f_2+\widetilde{\mathcal{D}}_0^\top\Gamma_1f_3\right)+L_{21}^\top \Gamma_{1} f_2\\
	&+L_{31}^\top \Gamma_1 f_3-\left(\Gamma_1 L_{14}+\Gamma_2 L_{14}-L_{32}^\top\right) \Gamma_1 f_3+\Gamma_1 f_1+\Gamma_2 f_1+f_4  \bigg\} dt,\\
    \mathbb{E}\Phi(T)=&\ \mathrm{O}_{3 \times 3}.
\end{aligned}\right.
\end{equation}

\begin{remark}
We note that (\ref{Gamma1 equation}) is symmetric Riccati equation, which admits unique solution by applying the Theorem 3.7 of \cite{Li-Nie-Wu-2023}. But (\ref{Gamma2 equation}) is an asymmetric one, which is hard to solve in general. In this paper, we need to assume the solvability of (\ref{Gamma2 equation}). Then noticing the equation (\ref{equation of E Phi}) is linear BODE, which obviously admits a unique solution.
\end{remark}

Therefore, the feedback decentralized strategy of the leader $\mathcal{A}_0$ can be given in the following.

\begin{thm}\label{theorem feedback of leader}
Let Assumptions \ref{A2.1}, \ref{A2.2}, \ref{A3.1},  \ref{Assumption FBODE} and \ref{A3.2} hold. Let $\Gamma_1(\cdot)$ be the solution to (\ref{Gamma1 equation}), and suppose that (\ref{Gamma2 equation}) admits a unique solution $\Gamma_2(\cdot)$, then the feedback decentralized strategy of the leader $\mathcal{A}_0$ can be represented as
\begin{equation}\label{leader new feedback strategy}
\begin{aligned}
u_0^*=&-\mathcal{R}_0^{-1}\left[\left(\mathcal{B}_0^\top \Gamma_1+\mathcal{D}_0^\top \Gamma_1 L_{21}+\widetilde{\mathcal{D}}^\top_0 \Gamma_1 L_{31}\right) \check{X}
       +\left(\mathcal{B}_0^\top \Gamma_2+\widetilde{\mathcal{D}}_0^\top \Gamma_1 L_{32}\right.\right.\\
      &\qquad\quad \left.+\widetilde{\mathcal{D}}_0^\top \Gamma_1 L_{14}^\top \Gamma_1+\widetilde{\mathcal{D}}_0^\top \Gamma_1 L_{14}^\top \Gamma_2\right) \mathbb{E} X
       +\left(\mathcal{B}_0^\top+\widetilde{\mathcal{D}}_0^\top \Gamma_1 L_{14}^\top\right) \mathbb{E}\Phi\\
      &\qquad\quad \left. +\mathcal{D}_0^\top \Gamma_1 f_2+\widetilde{\mathcal{D}}_0^\top\Gamma_1 f_3\right],\quad \mathbb{P}\text{-a.s.},\ \text{a.e.}\ t\in[0,T],
\end{aligned}
\end{equation}
where $\check{X}(\cdot)\in L_{\mathcal{F}_t^{W_0}}^2\left(0, T ; \mathbb{R}^3\right)$ satisfied the following SDE:
\begin{equation}\label{check X}
\left\{
\begin{aligned}
 d \check{X}= & \left\{\left[L_{11}-\left(\mathcal{B}_0-L_{14}\Gamma_1\widetilde{\mathcal{D}}_0\right)\mathcal{R}_0^{-1}\left(\mathcal{B}_0^\top \Gamma_1
               +\mathcal{D}_0^\top \Gamma_1 L_{21}+\widetilde{\mathcal{D}}^\top_0 \Gamma_1 L_{31}\right) \right]\check{X}\right. \\
              &+\left[L_{12}-\left(L_{13}+L_{14}\Gamma_1L_{14}^\top\right)\left(\Gamma_1+\Gamma_2\right)-L_{14}\Gamma_1\left(L_{31}+L_{32}\right)\right. \\
              &\left.-\left(\mathcal{B}_0-L_{14}\Gamma_1\widetilde{\mathcal{D}}_0\right)\mathcal{R}_0^{-1}\left(\mathcal{B}_0^\top \Gamma_2+\widetilde{\mathcal{D}}_0^\top \Gamma_1 L_{32}
              +\widetilde{\mathcal{D}}_0^\top \Gamma_1 L_{14}^\top \Gamma_1+\widetilde{\mathcal{D}}_0^\top \Gamma_1 L_{14}^\top \Gamma_2\right)\right]\mathbb{E}X \\
              & \left.-\left[L_{13}+L_{14}\Gamma_1L_{14}^\top+\left(\mathcal{B}_0-L_{14}\Gamma_1\widetilde{\mathcal{D}}_0\right)\mathcal{R}_0^{-1}\left(\mathcal{B}_0^\top
               +\widetilde{\mathcal{D}}_0^\top \Gamma_1 L_{14}^\top\right)\right]\mathbb{E}\Phi\right. \\
              &\left.+f_1-L_{14}\Gamma_1f_3-\left(\mathcal{B}_0-L_{14}\Gamma_1\widetilde{\mathcal{D}}_0\right)\mathcal{R}_0^{-1}\left(\mathcal{D}_0^\top \Gamma_1 f_2
               +\widetilde{\mathcal{D}}_0^\top\Gamma_1 f_3\right)\right\} d t \\
              &+\left\{\left[L_{21}-\mathcal{D}_0\mathcal{R}_0^{-1}\left(\mathcal{B}_0^\top \Gamma_1
               +\mathcal{D}_0^\top \Gamma_1 L_{21}+\widetilde{\mathcal{D}}^\top_0 \Gamma_1 L_{31}\right)\right]\check{X}\right.\\
              &\left.-\mathcal{D}_0\mathcal{R}_0^{-1}\left(\mathcal{B}_0^\top \Gamma_2+\widetilde{\mathcal{D}}_0^\top \Gamma_1 L_{32}
              +\widetilde{\mathcal{D}}_0^\top \Gamma_1 L_{14}^\top \Gamma_1+\widetilde{\mathcal{D}}_0^\top \Gamma_1 L_{14}^\top \Gamma_2\right)\mathbb{E}X\right. \\
              &\left.-\mathcal{D}_0\mathcal{R}_0^{-1}\left(\mathcal{B}_0^\top+\widetilde{\mathcal{D}}_0^\top \Gamma_1 L_{14}^\top\right) \mathbb{E}\Phi
               -\mathcal{D}_0\mathcal{R}_0^{-1}\left(\mathcal{D}_0^\top \Gamma_1 f_2
               +\widetilde{\mathcal{D}}_0^\top\Gamma_1 f_3\right)  \right\}dW_0, \\
\check{X}(0)= &\ \Xi,
\end{aligned}
\right.
\end{equation}
$\mathbb{E}{X}(\cdot)\in\mathbb{R}^3$ satisfies the following ODE:
\begin{equation}\label{E X}
\left\{
\begin{aligned}
d \mathbb{E}X= & \left\{\left[L_{11}+L_{12}-\left(L_{13}+L_{14}\Gamma_1L_{14}^\top\right)\left(\Gamma_1+\Gamma_2\right)-L_{14}\Gamma_1\left(L_{31}+L_{32}\right)\right.\right.\\
              &-\left(\mathcal{B}_0-L_{14}\Gamma_1\widetilde{\mathcal{D}}_0\right)\mathcal{R}_0^{-1}\left(\mathcal{B}_0^\top \Gamma_1+\mathcal{B}_0^\top \Gamma_2
               +\mathcal{D}_0^\top \Gamma_1 L_{21}+\widetilde{\mathcal{D}}^\top_0 \Gamma_1 L_{31} \right. \\
              &\left.\left.+\widetilde{\mathcal{D}}_0^\top \Gamma_1 L_{32}
              +\widetilde{\mathcal{D}}_0^\top \Gamma_1 L_{14}^\top \Gamma_1+\widetilde{\mathcal{D}}_0^\top \Gamma_1 L_{14}^\top \Gamma_2\right)\right]\mathbb{E}X \\
              & \left.-\left[L_{13}+L_{14}\Gamma_1L_{14}^\top+\left(\mathcal{B}_0-L_{14}\Gamma_1\widetilde{\mathcal{D}}_0\right)\mathcal{R}_0^{-1}\left(\mathcal{B}_0^\top
               +\widetilde{\mathcal{D}}_0^\top \Gamma_1 L_{14}^\top\right)\right]\mathbb{E}\Phi\right. \\
              &\left.+f_1-L_{14}\Gamma_1f_3-\left(\mathcal{B}_0-L_{14}\Gamma_1\widetilde{\mathcal{D}}_0\right)\mathcal{R}_0^{-1}\left(\mathcal{D}_0^\top \Gamma_1 f_2
               +\widetilde{\mathcal{D}}_0^\top\Gamma_1 f_3\right)\right\} d t, \\
\mathbb{E}X(0)= &\ \Xi,
\end{aligned}
\right.
\end{equation}
and $\mathbb{E}\Phi(\cdot)\in\mathbb{R}^3$ satisfies the BODE (\ref{equation of E Phi}), respectively.
\end{thm}
\begin{proof}
(\ref{leader new feedback strategy}) can be obtained by (\ref{leader feedback strategy}) with $\check{\Phi}(t)\equiv\mathbb{E}\Phi(t),t\in[0,T]$. (\ref{check X}) can be achieved from the first equation of $X$ in (\ref{dimension expansion}), noting (\ref{Phi equation}), the third equality in (\ref{check Y Z}) and (\ref{leader new feedback strategy}). And (\ref{E X}) can be got by (\ref{check X}) directly. The proof is complete.
\end{proof}

Similar to the subsection 3.2, in brief, we discuss the well-posedness of (\ref{dimension expansion}), which is equivalent to (\ref{leader Hamiltonian system}). For simplicity of notation, by (\ref{leader new feedback strategy}) we denote
\begin{equation}\label{leader new feedback strategy simple}
u_0^*=\theta_1 \check{X}+\theta_2 \mathbb{E} X +\theta_3 \mathbb{E}\Phi+\theta_4,\quad \mathbb{P}\text{-a.s.},\ \text{a.e.}\ t\in[0,T],
\end{equation}
where
$$
\left\{\begin{aligned}
\theta_1:=&-\mathcal{R}_0^{-1}\left(\mathcal{B}_0^\top \Gamma_1+\mathcal{D}_0^\top \Gamma_1 L_{21}+\widetilde{\mathcal{D}}^\top_0 \Gamma_1 L_{31}\right),\\
\theta_2:=&-\mathcal{R}_0^{-1}\left(\mathcal{B}_0^\top \Gamma_2+\widetilde{\mathcal{D}}_0^\top \Gamma_1 L_{32}+\widetilde{\mathcal{D}}_0^\top \Gamma_1 L_{14}^\top \Gamma_1+\widetilde{\mathcal{D}}_0^\top \Gamma_1 L_{14}^\top \Gamma_2\right),\\
\theta_3:=&-\mathcal{R}_0^{-1}\left(\mathcal{B}_0^\top+\widetilde{\mathcal{D}}_0^\top \Gamma_1 L_{14}^\top\right),\\
\theta_4:=&-\mathcal{R}_0^{-1}\left(\mathcal{D}_0^\top \Gamma_1 f_2+\widetilde{\mathcal{D}}_0^\top\Gamma_1 f_3\right).
\end{aligned}\right.
$$

Noticing (\ref{check Y Z}) and (\ref{leader new feedback strategy}), the first equation of (\ref{dimension expansion}) becomes
\begin{equation}\label{X explicit}
\left\{\begin{aligned}
	d X= & \left\{L_{11} X+\left[L_{12} -L_{13}\left(\Gamma_1+\Gamma_2)-L_{14}(\Gamma_1L_{31}+\Gamma_1L_{32}+\Gamma_1L_{14}^\top\Gamma_1+\Gamma_1L_{14}^\top\Gamma_2\right)\right]\mathbb{E} X\right. \\
	&-\left(L_{13}+L_{14}\Gamma_1L_{14}^\top\right) \mathbb{E} \Phi+L_{14}\Gamma_1\widetilde{\mathcal{D}}_0\left[\left(\theta_1 +\theta_2\right) \mathbb{E} X +\theta_3 \mathbb{E}\Phi+\theta_4\right] \\
	&+\mathcal{B}_0 \left(\theta_1 \check{X}+\theta_2 \mathbb{E} X +\theta_3 \mathbb{E}\Phi+\theta_4\right)+f_1\Big\} d t \\
	& +\left[L_{21} X+\mathcal{D}_0 \left(\theta_1 \check{X}+\theta_2 \mathbb{E} X +\theta_3 \mathbb{E}\Phi+\theta_4\right)+f_2\right] d W_0 \\
	&+\left[L_{31} X +\left(L_{32}+L_{14}^\top\Gamma_1+L_{14}^\top\Gamma_2\right) \mathbb{E} X+\widetilde{\mathcal{D}}_0 \left(\theta_1 \check{X}+\theta_2 \mathbb{E} X +\theta_3 \mathbb{E}\Phi+\theta_4\right)+f_3\right] d W ,\\
	X(0)= &\ \Xi.
\end{aligned}\right.
\end{equation}
Similar as subsection 3.2, we can obtain the solvability of (\ref{X explicit}). And then, we can obtain the solvability of the second equation of (\ref{dimension expansion}). The details of the proof are left to the interested readers.

\section{$\varepsilon$-Stackelberg-Nash equilibria analysis}

In Section 3, we derived the decentralized Stackelberg-Nash equilibrium, denoted as $\left(u_0^*(\cdot),u^*(\cdot)\right.$ $\left.\equiv\left(u^*_1(\cdot), u^*_2(\cdot), \ldots, u^*_N(\cdot)\right)\right)$, for the limiting Problem \ref{problem decentralized}. In this section, we will prove that the decentralized strategies $\left(u_0^*(\cdot),u^*(\cdot)\right)$ satisfy the $\varepsilon$-Stackelberg-Nash equilibrium property, for Problem \ref{problem centralized}. To begin with, let's define what an $\varepsilon$-Stackelberg-Nash equilibrium entails.

\newtheorem{definition}{Definition}[section]
\begin{definition}\label{varipsilon Stackelberg-Nash equalibrium}
A strategy $u^*(\cdot)\equiv\left(u_0^*(\cdot), u_1^*(\cdot), \cdots, u_N^*(\cdot)\right)$ is said to be an $\varepsilon$-Stackelberg-Nash equilibrium for Problem \ref{problem centralized}, if there exists a non-negative value $\varepsilon = \varepsilon(N)$ such that $\lim_{N \rightarrow \infty} \varepsilon(N) = 0$, which satisfies:

(i) for given $u_0(\cdot) \in \mathcal{U}_{0}^{l,d}$, for $i=1, \cdots, N$, the control stategy $u_i^*(\cdot)$ is a mapping $u_i^*:\mathcal{U}_0^{l,d}\rightarrow\mathcal{U}_i^{f,d}$ satisfying
\begin{equation*}
\mathcal{J}_i\left(u^*_i[u_0](\cdot), u^*_{-i}[u_0](\cdot)\right) \leq  \mathcal{J}_i\left(u_i(\cdot), u^*_{-i}[u_0](\cdot)\right)+\varepsilon, \text{ for any $u_i(\cdot)\in \mathcal{U}_i^{f,d}$};
\end{equation*}

(ii) the control strategy  $u^*_0(\cdot)$ of the leader $\mathcal{A}_0$ satisfies,
\begin{equation*}
\mathcal{J}_0\left(u^*_0(\cdot)\right)\leq  \mathcal{J}_0\left(u_0(\cdot )\right)+\varepsilon,\text{ for any $u_0(\cdot) \in \mathcal{U}_{0}^{l,d}$}.
\end{equation*}
\end{definition}

The main result of this section is the following theorem.

\begin{thm}\label{thm4.1}
Let Assumptions \ref{A2.1}, \ref{A2.2}, \ref{A3.1} and   \ref{Assumption FBODE} hold, then the set of strategies $(u_0^*(\cdot), u_1^*(\cdot),\\ u_2^*(\cdot), \ldots, u_N^*(\cdot)$ constitutes an $\varepsilon$-Stackelberg-Nash equilibrium for Problem \ref{problem centralized}, where
\begin{equation}\label{varipsilon Stackelberg-Nash equalibrium ui* and u0*}
\left\{\begin{aligned}
u_i^*&=R_1^{-1}\left(B_1 \hat{p}_i+D_1 \hat{q}_i+\widetilde{D}_1 \hat{\widetilde{q}}_i\right),\quad i=1,2,\ldots, N, \\
u_0^*&=R_0^{-1}\left(B_0 \check{y}_0+D_0 \check{z}_0+\widetilde{D}_0 \check{\widetilde{z}}_0\right),
\end{aligned}\right.
\end{equation}
with $\left(p_i(\cdot),q_i(\cdot),\widetilde{q}_i(\cdot)\right)$, $\left(y_0(\cdot),z_0(\cdot),\widetilde{z}_0(\cdot)\right)$ satisfying (\ref{follower Hamiltonian system}), (\ref{leader Hamiltonian system}), respectively.
\end{thm}

To establish this theorem, we shall present several lemmas.

\newtheorem{lemma}{Lemma}[section]
\begin{lemma}\label{CC lemma}
Let Assumptions \ref{A2.1}, \ref{A2.2}, \ref{A3.1} and   \ref{Assumption FBODE} hold, we have the following estimation:
\begin{equation}\label{CC}
\mathbb{E} \left[\sup _{0 \leq t \leq T}\left|{x^*}^{(N)}(t)-z(t)\right|^2\right] = O\left(\frac{1}{N}\right),
\end{equation}
\begin{equation}\label{follower approximation}
\sup _{1 \leq i \leq N} \mathbb{E} \left[\sup _{0 \leq t \leq T}\left|x_i^*(t)-\bar{x}_i^*(t)\right|^2\right] = O\left(\frac{1}{N}\right),
\end{equation}
\begin{equation}\label{leader approximation}
 \mathbb{E} \left[\sup _{0 \leq t \leq T}\left|x_0(t)-\bar{x}_0(t)\right|^2\right] = O\left(\frac{1}{N}\right).
\end{equation}
\end{lemma}

\begin{proof}
From (\ref{follower state}) and (\ref{z state}), we obtain
\begin{equation*}
\left\{\begin{aligned}
d\left({x^*}^{(N)}-z\right)  =&\left[\left(A_1+E_1\right)\left(x^{*(N)}-z\right)+B_1\left(\frac{1}{N} \sum_{i=1}^N u_i^*-\mathbb{E} u_i^*\right)\right] d t \\
& +\frac{1}{N} \sum_{i=1}^N\left[C_1 x_i^*+D_1 u_i^*+F_1 x^{(N)}+\sigma_1\right] d W_i \\
& +\left[\left(\widetilde{C}_1+\widetilde{F}_1\right)\left(x^{*(N)}-z\right)+\widetilde{D}_1\left(\frac{1}{N} \sum_{i=1}^N u_i^*-\mathbb{E} u_i^*\right)\right] d W ,\\
\left({x^*}^{(N)}-z\right)(0) =&\ 0.
\end{aligned}\right.
\end{equation*}
By using BDG's inequality and noting the boundedness condition of the coefficients in Assumption \ref{A2.1}, we have ($K$ is some positive constant and may be different line by line),
\begin{equation*}
\begin{aligned}
& \mathbb{E}\left[\sup _{0\leq t \leq  T}\left|{x^*}^{(N)}(t)-z(t)\right|^2\right]\\
& =  \mathbb{E}\left[ \sup _{0\leq t \leq T}\left|\int_{0}^{t}\left(A_1+E_1\right)\left({x^*}^{(N)}-z\right)+B_1\left(\frac{1}{N}\sum_{i=1}^{N}u_i-\mathbb{E}u_i^*\right)ds \right.\right.\\
&\qquad\ +\int_0^t\frac{1}{N} \sum_{i=1}^N\left[C_1 x_i^*+D_1 u_i^*+F_1 x^{(N)}+\sigma_1\right] d W_i\\
&\qquad\ \left.\left.\left.+\int_0^t\left[\left(\widetilde{C}_1+\widetilde{F}_1\right)\left(x^{*(N)}-z\right)\right.+\widetilde{D}_1\left(\frac{1}{N} \sum_{i=1}^N u_i^*-\mathbb{E} u_i^*\right)\right] d W\right|^2\right]\\
&\leq K\left\{ \mathbb{E}\left[\sup _{0 \leq t \leq T} \left|\int_0^t \left(  {x^*}^{(N)}-z\right) +\left(\frac{1}{N} \sum_{i=1}^N u_i^*-\mathbb{E} u_i^*\right) d s \right|^2 \right]\right.\\
&\qquad\ +\mathbb{E}\int_0^T\frac{1}{N^2} \sum_{i=1}^N\left|C_1 x_i^*+D_1 u_i^*+F_1 {x^*}^{(N)}+\sigma_1\right|^2 d t\\
&\qquad\ \left.+\mathbb{E}\int_0^T\left|\left(\widetilde{C}_1+\widetilde{F}_1\right)\left(x^{*(N)}-z\right)+\widetilde{D}_1\left(\frac{1}{N} \sum_{i=1}^N u_i^*-\mathbb{E} u_i^*\right)\right|^2 dt\right\}\\
&\leq K\left[ \mathbb{E}\int_0^T \left(\sup _{0 \leq t \leq T} \left|{x^*}^{(N)}-z\right|^2 +\left|\frac{1}{N} \sum_{i=1}^N u_i^*-\mathbb{E} u_i^*\right|^2\right) d t\right.\\
&\qquad\ \left.+\mathbb{E}\int_0^T\frac{1}{N^2} \sum_{i=1}^N\left|C_1 x_i^*+D_1 u_i^*+F_1 {x^*}^{(N)}+\sigma_1\right|^2 d t\right].
\end{aligned}
\end{equation*}
From (\ref{follower state}), by applying Gronwall's inequality, one can prove $\mathbb{E}\left[ \sup _{0 \leq t \leq T} \sum_{i=1}^N\left|x^*_i(t)\right|^2\right]=O(N)$, and thus $\mathbb{E}\left[ \sup _{0 \leq t \leq T}\left|\bar{x}_i(t)\right|^2\right] \leq K$. We can also show $\mathbb{E}\left[ \sup _{0<t<T}|z(t)|^2\right] \leq K$ by noticing (\ref{z state}). Then, we have
\begin{equation*}
\begin{aligned}
&\mathbb{E}\int_0^T\frac{1}{N^2} \sum_{i=1}^N\left|C_1 x_i^*+D_1 u_i^*+F_1 {x^*}^{(N)}+\sigma_1\right|^2 d t  \\
&=\mathbb{E}\int_0^T\frac{1}{N^2} \sum_{i=1}^N\left|C_1 x_i^*+D_1 u_i^*+F_1 \left({x^*}^{(N)}-z\right)+F_1 z+\sigma_1\right|^2 d t \\
&\leq\frac{K}{N^2}\mathbb{E}\int_{0}^{T}\sum_{i=1}^N\left|1+\left({x^*}^{(N)}-z\right)\right|^2dt
\leq\frac{K}{N}\mathbb{E}\int_0^T\sum_{i=1}^N\left|1+\left({x^*}^{(N)}-z\right)\right|^2dt.
\end{aligned}
\end{equation*}
Moreover, noticing that $u^*_i$ and $u^*_j$ are i.i.d., we get
\begin{equation*}
	\mathbb{E} \int_0^T \left|\frac{1}{N} \sum_{i=1}^N u_i^*-\mathbb{E} u_i^*\right|^2 d t =\frac{1}{N^2}\sum_{i=1}^N \mathbb{E} \int_0^t \left| u_i^*-\mathbb{E} u_i^*\right|^2 d t
	\leq \frac{K}{N} =O\left(\frac{1}{N}\right).
\end{equation*}
Therefore,
\begin{equation}\label{Gronwall}
 \mathbb{E} \left[\sup _{0\leq t \leq T}\left|{x^*}^{(N)}(t)-z(t)\right|^2\right]\leq K \mathbb{E} \left(\int_0^T \sup _{0 \leq t \leq T}\left|{x^*}^{(N)}(t)-z(t)\right|^2 +1\right) d t+O\left(\frac{1}{N}\right).
\end{equation}
From (\ref{Gronwall}), by applying Gronwall's inequality, (\ref{CC}) holds. Then, by subtracting (\ref{follower limiting state}) from (\ref{follower state}), we have
\begin{equation*}
	\left\{\begin{aligned}
		d\left(x_i^*-\bar{x}_i^*\right)  =&\left[A_1\left(x_i^*-\bar{x}_i^*\right) +E_1\left(x^{*(N)}-z\right)\right] d t+\left[C_1\left(x_i^*-\bar{x}_i^*\right) +F_1\left(x^{*(N)}-z\right)\right] d W_i \\
		& +\left[\widetilde{C}_1\left(x_i^*-\bar{x}_i^*\right) +\widetilde{F}_1\left(x^{*(N)}-z\right)\right] d W \\
		\left(x_i^*-\bar{x}_i^*\right) (0) =&\ 0,\text{ for }i=1,\cdots, N.
	\end{aligned}\right.
\end{equation*}
Then, by using BDG's inequality and noting the boundedness condition of the coefficients in Assumption \ref{A2.1}, we have ($K$ is some positive constant and may be different line by line),
\begin{equation*}
	\begin{aligned}
		& \mathbb{E}\left[\sup _{0\leq t \leq  T}\left|x_i^*-\bar{x}_i^*\right|^2\right]
        =  \mathbb{E}\left[ \sup _{0\leq t \leq T}\left|\int_0^t\left[A_1\left(x_i^*-\bar{x}_i^*\right) +E_1\left(x^{*(N)}-z\right)\right]ds \right.\right.\\
		&\quad \left.\left.+\int_0^t\left[C_1\left(x_i^*-\bar{x}_i^*\right) +F_1\left(x^{*(N)}-z\right)\right] d W_i
         +\int_0^t\left[\widetilde{C}_1\left(x_i^*-\bar{x}_i^*\right) +\widetilde{F}_1\left(x^{*(N)}-z\right)\right] d W\right|^2\right]\\
		&\leq K\left\{ \mathbb{E}\left[\sup _{0 \leq t \leq T} \left|\int_0^t \left[A_1\left(x_i^*-\bar{x}_i^*\right) +E_1\left(x^{*(N)}-z\right)\right] d s \right|^2\right] \right.\\
		&\qquad\quad +\mathbb{E}\left[\sup _{0\leq t \leq  T}\left|\int_{0}^{t}\left[C_1\left(x_i^*-\bar{x}_i^*\right) +F_1\left(x^{*(N)}-z\right)\right] d W_i\right|^2\right]\\
		&\qquad\quad \left.+\mathbb{E}\left[\sup _{0\leq t \leq  T}\left|\int_{0}^{t}\left[\widetilde{C}_1\left(x_i^*-\bar{x}_i^*\right) +\widetilde{F}_1\left(x^{*(N)}-z\right)\right] d W\right|^2\right]\right\}\\
		&\leq K\left[\mathbb{E} \int_0^T\left|A_1\left(x_i^*-\bar{x}_i^*\right) +E_1\left(x^{*(N)}-z\right)\right|^2dt +\mathbb{E} \int_0^T\left|C_1\left(x_i^*-\bar{x}_i^*\right) +F_1\left(x^{*(N)}-z\right)\right|^2 d t\right.\\
		&\qquad\ \left.+\mathbb{E} \int_0^T\left|\widetilde{C}_1\left(x_i^*-\bar{x}_i^*\right) +\widetilde{F}_1\left(x^{*(N)}-z\right)\right|^2 d W\right]\\
		&\leq K\left\{ \int_0^T \mathbb{E}\left[\sup _{0 \leq s \leq t} \left|x_i^*-\bar{x}_i^*\right|^2\right]dt+\int_0^T \mathbb{E}\left[\sup _{0 \leq s \leq t} \left|x^{*(N)}-z\right|^2\right]dt\right\}.
	\end{aligned}
\end{equation*}
Noting (\ref{CC}), we can using Gronwall's inequality to obtain
\begin{equation*}
	\mathbb{E} \left[\sup _{0 \leq t \leq T}\left|x_i^*(t)-\bar{x}_i^*(t)\right|^2\right] = O\left(\frac{1}{N}\right).
\end{equation*}
Due to the arbitrariness of $i$, we have (\ref{follower approximation}). The proof of (\ref{leader approximation}) is similar to (\ref{follower approximation})  and thus we omit the proof. And then, the proof is complete.
\end{proof}

\begin{lemma}\label{left half of follower NE}
Let Assumptions \ref{A2.1}, \ref{A2.2}, \ref{A3.1} and   \ref{Assumption FBODE} hold, we have the following estimations:
\begin{equation}\label{estimate followers cost}
\left|\mathcal{J}_i\left(u_i^*(\cdot), u_{-i}^*(\cdot)\right)-J_i\left(u_i^*(\cdot)\right) \right|=O\left(\frac{1}{\sqrt{N}}\right),
\end{equation}
\begin{equation}\label{estimate leader cost}
\left|\mathcal{J}_0\left(u_0^*(\cdot)\right)-J_0\left(u_0^*(\cdot)\right) \right|=O\left(\frac{1}{\sqrt{N}}\right).
\end{equation}
\end{lemma}

\begin{proof}
From (\ref{follower cost}) and (\ref{follower limiting cost}), we get
\begin{equation*}
\begin{aligned}
&\left|\mathcal{J}_i\left(u_i^*(\cdot), u_{-i}^*(\cdot)\right)-J_i\left(u_i^*(\cdot)\right)\right| \\
= &\ \frac{1}{2} \mathbb{E}\left\{\int_0^T\left\{ Q_1\left[x_i^*-\left(\lambda x_0+(1-\lambda) {x^*}^{(N)}\right)\right]^2 -Q_1\left[\bar{x}_i^*-\left(\lambda \bar{x}_0+(1-\lambda) z\right)\right]^2\right\} dt\right. \\
&  +G_1\left[x_i^{*2}(T)-\bar{x}_i^{*2}(T)\right]\bigg\} .
\end{aligned}
\end{equation*}
Noting Lemma \ref{CC lemma}, we can obtain
\begin{equation*}
\begin{aligned}
\frac{1}{2}& \mathbb{E}\int_0^T \left\{Q_1\left[x_i^*-\left(\lambda x_0+(1-\lambda) {x^*}^{(N)}\right)\right]^2 -Q_1\left[\bar{x}_i^*-\left(\lambda \bar{x}_0+(1-\lambda) z\right)\right]^2\right\} dt \\
=&\ \frac{1}{2} \mathbb{E}\int_0^T \left\{Q_1\left[x_i^*-\bar{x}_i^* -\lambda \left(x_0-\bar{x}_0\right)-(1-\lambda)\left(x^{*(N)}-z\right)\right]^2 \right.\\
&\qquad +2\left[\bar{x}_i^*-\left(\lambda \bar{x}_0+(1-\lambda) z\right)\right]\left[x_i^*-\bar{x}_i^* -\lambda \left(x_0-\bar{x}_0\right)-(1-\lambda)\left(x^{*(N)}-z\right)\right]\bigg\} dt\\
\leq&\ \frac{1}{2} \mathbb{E}\int_0^T\left\{ K\left(\left|x_i^*-\bar{x}_i^*\right|^2+\left|x_0-\bar{x}_0\right|^2 + \left|x^{*(N)}-z\right|^2\right)
 +\left(\int_0^{T}\left[\bar{x}_i^*-\left(\lambda \bar{x}_0+(1-\lambda) z\right)\right]^2 d t\right)^{\frac{1}{2}} \right.\\
&\qquad \left.\times\left(\int_0^T\left[x_i^*-\bar{x}_i^* -\lambda \left(x_0-\bar{x}_0\right)-(1-\lambda)\left(x^{*(N)}-z\right)\right] ^2 \right) ^{\frac{1}{2}}\right\} d t\\
\end{aligned}
\end{equation*}
\begin{equation*}
\begin{aligned}
\leq&\ K  \mathbb{E}\int_0^T \left(\left|x_i^*-\bar{x}_i^*\right|^2+\left|x_0-\bar{x}_0\right|^2 + \left|x^{*(N)}-z\right|^2\right) d t\\
&+K\mathbb{E}\left[\int_0^T \left|x_i^*-\bar{x}_i^*\right|^2 dt \right]^{\frac{1}{2}}+K\mathbb{E}\left[\int_0^T \left|x_0^*-\bar{x}_0^*\right|^2 dt \right]^{\frac{1}{2}}
 +K\mathbb{E}\left[\int_0^T \left|x^{*(N)}-z\right|^2 dt \right]^{\frac{1}{2}}\\
=&\ O\left(\frac{1}{\sqrt{N}}\right).
\end{aligned}
\end{equation*}
Similarly, the order of the terminal term is also $O\left(\frac{1}{\sqrt{N}}\right)$. Therefore, (\ref{estimate followers cost}) is proved. And then,
\begin{equation*}
\begin{aligned}
&\left|\mathcal{J}_0\left(u_0^*(\cdot)\right)-J_0\left(u_0^*(\cdot)\right) \right| \\
= &\ \frac{1}{2} \mathbb{E}\left\{\int_0^T\left[ Q_0\left(x_0^*- {x^*}^{(N)}\right)^2 -Q_0\left(\bar{x}_0^*- z\right)^2\right] dt
 +G_1\left[x_0^{*2}(T)-\bar{x}_0^{*2}(T)\right]\right\} .
\end{aligned}
\end{equation*}
Using similar method, we can prove (\ref{estimate leader cost}). The proof is complete.
\end{proof}

Now, let's consider the perturbed controls of the followers. Suppose the $i$-th follower $\mathcal{A}_i$ takes the value of $u_i(\cdot)$, while $\mathcal{A}_j(j \neq i)$ still takes $u^*_j(\cdot)$. Then, the centralized state equation for the perturbation of $\mathcal{A}_i$ is
\begin{equation}\label{perturbation x-i}
	\left\{\begin{aligned}
		d x_{-i} =&\left(A_1 x_{-i}+B_1 u_i+E_1 x_-^{(N)}+b_1\right) d t \\
		& +\left(C_1 x_{-i}+D_1 u_i+F_1 x_-^{(N)}+\sigma_1\right) d W_i \\
		& +\left(\widetilde{C}_1 x_{-i}+\widetilde{D}_1 u_i+\widetilde{F}_1 x_-^{(N)}+\widetilde{\sigma}_1\right) d W, \\
		x_{-i}(0) =&\ \xi,
	\end{aligned}\right.
\end{equation}
where $x_-^{(N)}(\cdot) := \frac{1}{N} \left(x_{-i}(\cdot)+\sum_{j=1,j\neq i}^N x_j(\cdot)\right)$. The centralized state equation for the $j$-th follower $\mathcal{A}_j$ is the following equation
\begin{equation}\label{perturbation xj}
	\left\{\begin{aligned}
		d x_j =&\left(A_1 x_j+B_1 u^*_j+E_1 x_-^{(N)}+b_1\right) d t \\
		& +\left(C_1 x_j+D_1 u^*_j+F_1 x_-^{(N)}+\sigma_1\right) d W_j \\
		& +\left(\widetilde{C}_1 x_j+\widetilde{D}_1 u^*_j+\widetilde{F}_1 x_-^{(N)}+\widetilde{\sigma}_1\right) d W, \\
		x_j(0) =&\ \xi.
	\end{aligned}\right.
\end{equation}
The corresponding decentralized states satisfy
\begin{equation}\label{perturbation bar x-i}
	\left\{\begin{aligned}
		d \bar{x}_{-i} =&\left(A_1 \bar{x}_{-i}+B_1 u_i+E_1 z+b_1\right) d t \\
		& +\left(C_1 \bar{x}_{-i}+D_1 u_i+F_1 z+\sigma_1\right) d W_i \\
		& +\left(\widetilde{C}_1 \bar{x}_{-i}+\widetilde{D}_1 u_i+\widetilde{F}_1 z+\widetilde{\sigma}_1\right) d W, \\
		\bar{x}_{-i}(0) =&\ \xi.
	\end{aligned}\right.
\end{equation}
and
\begin{equation}\label{perturbation bar xj}
	\left\{\begin{aligned}
		d \bar{x}_j  =&\left(A_1 \bar{x}_j+B_1 u^*_j+E_1 z+b_1\right) d t \\
		& +\left(C_1 \bar{x}_j+D_1 u^*_j+F_1 z+\sigma_1\right) d W_j \\
		& +\left(\widetilde{C}_1 \bar{x}_j+\widetilde{D}_1 u^*_j+\widetilde{F}_1 z+\widetilde{\sigma}_1\right) d W, \\
		\bar{x}_j(0) =&\ \xi.
	\end{aligned}\right.
\end{equation}
respectively.
And then, for the perturbation of the $i$-th follower $\mathcal{A}_i$, the corresponding centralized state equation of the leader transforms into
\begin{equation}\label{leader state perturbed by follower centralized}
\left\{\begin{aligned}
d x_{0;-i}= & \left(A_0 x_{0;-i}+B_0 u_0+E_1 x_-^{(N)}+b_1\right) d t \\
& +\left(C_0 x_{0;-i}+D_0 u_0+F_1 x_-^{(N)}+\sigma_1\right) d W_0 \\
& +\left(\widetilde{C}_0 x_{0;-i}+\widetilde{D}_0 u_0+\widetilde{F}_1 x_-^{(N)}+\widetilde{\sigma}_1\right) d W ,\\
x_{0;-i}(0)= &\ \xi_0,
\end{aligned}\right.
\end{equation}
together with the decentralized state equation of the leader
\begin{equation}\label{leader state perturbed by follower decentralized}
\left\{\begin{aligned}
d \bar{x}_{0;-i}= & \left(A_0 \bar{x}_{0;-i}+B_0 u_0+E_1 z+b_1\right) d t \\
& +\left(C_0 \bar{x}_{0;-i}+D_0 u_0+F_1 z+\sigma_1\right) d W_0 \\
& +\left(\widetilde{C}_0 \bar{x}_{0;-i}+\widetilde{D}_0 u_0+\widetilde{F}_1 z+\widetilde{\sigma}_1\right) d W ,\\
\bar{x}_{0;-i}(0)= &\ \xi_0,
\end{aligned}\right.
\end{equation}

By the definition (\ref{varipsilon Stackelberg-Nash equalibrium}) of $\varepsilon$-Stackelberg-Nash equilibria, we need to show
$$
\inf _{u_i(\cdot) \in\ \mathcal{U}_{i}^{f,d}} \mathcal{J}_i\left(u_i(\cdot), u^*_{-i}(\cdot)\right) \geq \mathcal{J}_i\left(u^*_i(\cdot), u^*_{-i}(\cdot)\right)-\varepsilon .
$$
Therefore, we only need to consider $u_i(\cdot)$ satisfies $\mathcal{J}_i\left(u_i(\cdot), u^*_{-i}(\cdot)\right) \leq \mathcal{J}_i\left(u^*_i(\cdot), u^*_{-i}(\cdot)\right)$. And then,
$$
\mathbb{E} \int_0^T R_0 u_i^2(t) d t \leq \mathcal{J}_i\left(u_i(\cdot), u^*_{-i}(\cdot)\right) \leq \mathcal{J}_i\left(u^*_i(\cdot), u^*_{-i}(\cdot)\right) \leq J_i\left(u^*_i(\cdot)\right)+O\left(\frac{1}{\sqrt{N}}\right),
$$
i.e.,
$$
\mathbb{E} \int_0^T\left|u_i(t)\right|^2 d t \leq K .
$$

\begin{lemma}\label{follower perturbed CC lemma}
Let Assumptions \ref{A2.1}, \ref{A2.2}, \ref{A3.1} and   \ref{Assumption FBODE} hold, we have the following estimations:
\begin{equation}\label{follower perturbed CC}
\begin{aligned}
& \mathbb{E} \left[\sup _{0 \leq t \leq T}\left|x_-^{(N)}(t)-z(t)\right|^2\right] =O\left(\frac{1}{N}\right), \\
& \sup _{1 \leq i \leq N} \mathbb{E} \left[\sup _{0 \leq t \leq T}\left|x_{-i}(t)-\bar{x}_{-i}(t)\right|^2\right] =O\left(\frac{1}{N}\right), \\
& \mathbb{E} \left[\sup _{0 \leq t \leq T}\left|x_{0;-i}(t)-\bar{x}_{0;-i}(t)\right|^2\right] =O\left(\frac{1}{N}\right).
\end{aligned}
\end{equation}
\end{lemma}

\begin{proof}
Utilizing a similar approach in the proof of Lemma \ref{CC lemma}, we can obtain
\begin{equation*}
\begin{aligned}
&\mathbb{E} \left[\sup_{0 \leq s \leq t}\left|x_-^{(N)}(s)-x^{*(N)}(s)\right|^2\right] \\
&\leq K \mathbb{E} \int_0^t \left(\sup _{0 \leq r \leq s}\left|x_-^{(N)}(r)-x^{*(N)}(r)\right|^2+\frac{1}{N^2} \left(u_i-u_i^*\right)^2(s)\right) d s +O\left(\frac{1}{N}\right).
\end{aligned}
\end{equation*}
By applying Gronwall's inequality, the first estimate in (\ref{follower perturbed CC}) holds.

Noticing Lemma \ref{CC lemma}, we have
\begin{equation*}
\begin{aligned}
& \mathbb{E} \left[\sup _{0 \leq t \leq T}\left|x_-^{(N)}(t)-z(t)\right|^2\right] =\mathbb{E} \left[\sup _{0 \leq t \leq T}\left|x_-^{(N)}(t)-x^{*(N)}(t)+x^{*(N)}(t)-z(t)\right|^2\right] \\
&\leq 2 \mathbb{E} \left[\sup _{0 \leq t \leq T}\left|x_-^{(N)}(t)-x^{*(N)}(t)\right|^2\right] + 2 \mathbb{E} \left[\sup _{0 \leq t \leq T}\left|x^{*(N)}(t)-z(t)\right|^2\right] =O\left(\frac{1}{N}\right).
\end{aligned}
\end{equation*}
Then, by standard estimates of SDEs, we can obtain the last two estimates in (\ref{follower perturbed CC}). The proof is complete.
\end{proof}

\begin{lemma}\label{right half of follower NE}
Let Assumptions \ref{A2.1}, \ref{A2.2}, \ref{A3.1} and   \ref{Assumption FBODE} hold, we have the following estimations:
\begin{equation}
\left|\mathcal{J}_i\left(u_i(\cdot), u_{-i}^*(\cdot)\right)-J_i\left(u_i(\cdot)\right) \right|=O\left(\frac{1}{\sqrt{N}}\right).
\end{equation}
\end{lemma}

\begin{proof}
From (\ref{follower cost}) and (\ref{follower limiting cost}), we get
\begin{equation*}
\begin{aligned}
&\left|\mathcal{J}_i\left(u_i(\cdot), u_{-i}^*(\cdot)\right)-J_i\left(u_i(\cdot)\right) \right| \\
= &\ \frac{1}{2} \mathbb{E}\left\{\int_0^T \left\{Q_1\left[x_{-i}-\left(\lambda x_{0;-i}+(1-\lambda) x_-^{(N)}\right)\right]^2 -Q_1\left[\bar{x}_i-\left(\lambda \bar{x}_{0;-i}+(1-\lambda) z\right)\right]^2\right\} dt\right. \\
&\qquad +G_1\left(x_{-i}^{2}(T)-\bar{x}_i^2(T)\right)\bigg\} .
\end{aligned}
\end{equation*}
Utilizing a similar approach in the proof of Lemma \ref{left half of follower NE}, we can complete the proof.
\end{proof}

The following is the proof of first part of Theorem \ref{thm4.1}.

\begin{proof}
Applying Lemma \ref{left half of follower NE} and Lemma \ref{right half of follower NE}, we obtain
\begin{equation*}
\mathcal{J}_i\left(u^*_i(\cdot), u^*_{-i}(\cdot)\right) \leq J_i\left(u^*_i(\cdot)\right)+O\left(\frac{1}{\sqrt{N}}\right) \leq  J_i\left(u_i(\cdot)\right)+O\left(\frac{1}{\sqrt{N}}\right) \leq \mathcal{J}_i\left(u_i(\cdot), u^*_{-i}(\cdot)\right)+O\left(\frac{1}{\sqrt{N}}\right).
\end{equation*}
\end{proof}

Next, let us consider a perturbed control $u_0(\cdot)$ for the leader $\mathcal{A}_0$, the corresponding centralized state equation is
\begin{equation}\label{perturbation m0}
\left\{\begin{aligned}
d x_{-0}= & \left(A_0 x_{-0}+B_0 u_0+E_0 x^{*(N)}+b_0\right) d t \\
& +\left(C_0 x_{-0}+D_0 u_0+F_0 x^{*(N)}+\sigma_0\right) d W_0 \\
& +\left(\widetilde{C}_0 x_{-0}+\widetilde{D}_0 u_0+\widetilde{F}_0 x^{*(N)}+\widetilde{\sigma}_0\right) d W ,\\
x_{-0}(0) =&\ \xi_0.
\end{aligned}\right.
\end{equation}
And the corresponding decentralized state equation is
\begin{equation}\label{perturbation bar m0}
\left\{\begin{aligned}
d \bar{x}_{-0}= & \left(A_0 \bar{x}_{-0}+B_0 u_0+E_0 z+b_0\right) d t \\
& +\left(C_0 \bar{x}_{-0}+D_0 u_0+F_0 z+\sigma_0\right) d W_0 \\
& +\left(\widetilde{C}_0 \bar{x}_{-0}+\widetilde{D}_0 u_0+\widetilde{F}_0 z+\widetilde{\sigma}_0\right) d W ,\\
\bar{x}_{-0}(0) =&\ \xi_0.
\end{aligned}\right.
\end{equation}

By the definition (\ref{varipsilon Stackelberg-Nash equalibrium}) of $\varepsilon$-Stackelberg-Nash equilibria, we need to show
$$
\inf _{u_0(\cdot) \in\ \mathcal{U}_{0}^{l,d}} \mathcal{J}_0\left(u_0(\cdot)\right) \geq \mathcal{J}_0\left(u^*_0(\cdot)\right)-\varepsilon.
$$
Therefore, we only need to consider $u_0(\cdot)$ satisfies $\mathcal{J}_0\left(u_0(\cdot)\right) \leq \mathcal{J}_0\left(u^*_0(\cdot)\right)$.
And then,
$$
\mathbb{E} \int_0^T R_0 {u_0}^2(t) d t \leqslant \mathcal{J}_0\left(u_0(\cdot)\right) \leq \mathcal{J}_0\left(u^*_0(\cdot)\right) \leq J_0\left(u^*_0(\cdot)\right)+O\left(\frac{1}{\sqrt{N}}\right) ,
$$
i.e..
$$
\mathbb{E} \int_0^T\left|u_0(t)\right|^2 d t \leqslant K .
$$

\begin{lemma}\label{leader perturbed CC lemma}
Let Assumptions \ref{A2.1}, \ref{A2.2}, \ref{A3.1} and   \ref{Assumption FBODE} hold, we have the following estimations:
\begin{equation}\label{leader perturbed CC}
\mathbb{E} \left[\sup _{0 \leq t \leq T}\left|x_{-0}(t)-\bar{x}_{-0}(t)\right|^2\right] =O\left(\frac{1}{N}\right) .
\end{equation}
\end{lemma}

\begin{proof}
First, we get
\begin{equation*}
\left\{\begin{aligned}
d \left(x_{-0}-\bar{x}_{-0}\right)= & \left[A_0\left( x_{-0}-\bar{x}_{-0}\right)+E_0\left( x^{*(N)}-z\right) \right] d t \\
& +\left[C_0\left( x_{-0}-\bar{x}_{-0}\right)+F_0\left( x^{*(N)}-z\right) \right] d W_0 \\
& +\left[\widetilde{C}_0\left( x_{-0}-\bar{x}_{-0}\right)+\widetilde{F}_0\left( x^{*(N)}-z\right) \right] d W ,\\
\left(x_{-0}-\bar{x}_{-0}\right)(0)=&\ 0 .
\end{aligned}\right.
\end{equation*}
Using standard estimate of SDEs and Lemma \ref{CC lemma}, (\ref{leader perturbed CC}) holds.
\end{proof}

\begin{lemma}\label{right half of leader NE}
Let Assumptions \ref{A2.1}, \ref{A2.2}, \ref{A3.1} and   \ref{Assumption FBODE} hold, we have the following estimations:
\begin{equation}
\left|\mathcal{J}_0\left(u_0(\cdot)\right)-J_0\left(u_0(\cdot)\right) \right|=O\left(\frac{1}{\sqrt{N}}\right).
\end{equation}
\end{lemma}

\begin{proof}
Since
\begin{equation*}
\begin{aligned}
&\left|\mathcal{J}_0\left(u_0(\cdot)\right)-J_0\left(u_0(\cdot)\right) \right| \\
= &\ \frac{1}{2} \mathbb{E}\left\{\int_0^T \left[ Q_0\left(x_{-0}- {x^*}^{(N)}\right)^2 -Q_0\left(\bar{x}_{-0}- z\right)^2\right] dt
 +G_1\left[x_{-0}^2(T)-\bar{x}_{-0}^2(T)\right]\right\},
\end{aligned}
\end{equation*}
using the same method as in the proof of Lemma \ref{left half of follower NE} and noticing Lemma \ref{CC lemma} and Lemma \ref{leader perturbed CC lemma}, we can prove (\ref{right half of leader NE}).
\end{proof}

The proof of the second part of Theorem \ref{thm4.1}.

\begin{proof}
Applying Lemma \ref{left half of follower NE} and Lemma \ref{right half of leader NE}, we obtain
\begin{equation*}
\mathcal{J}_0\left(u^*_0(\cdot)\right) \leq J_0\left(u^*_0(\cdot)\right)+O\left(\frac{1}{\sqrt{N}}\right) \leq  J_0\left(u_0(\cdot)\right)+O\left(\frac{1}{\sqrt{N}}\right) \leq \mathcal{J}_0\left(u_0(\cdot)\right)+O\left(\frac{1}{\sqrt{N}}\right).
\end{equation*}
\end{proof}

\section{Numerical examples}

In this section, we give an example with certain particular coefficients to demonstrate the effectiveness of our theoretical results. The simulation parameters are given as follows: $A_0=0.1,A_1=-2,B_0=B_1=5,E_0=\lambda=0.5,b_0=C_0=D_0=F_0=\sigma_0=\tilde{C}_0=\tilde{D}_0=\tilde{f}_0=\tilde{\sigma}_0=E_1=b_1=C_1=D_1=F_1=\sigma_1=\tilde{C}_1=\tilde{D}_1=\tilde{F}_1=\tilde{\sigma}_1=Q_0=R_0=Q_1=R_1=1,G_0=0.05,G_1=0.3,\xi=\xi_0=0.5,T=1$.
By the  Euler's method, we plot the solution curves of Riccati equations (\ref{RE P1}), (\ref{RE P2}), (\ref{Gamma1 equation}) and (\ref{Gamma2 equation})
in Figures \ref{fig:P} and \ref{fig:Gamma}. In detail, the curves of $P_1,P_2$ are one-dimensional functions, respectively. But the curves of $\Gamma_1,\Gamma_2$ are two $3 \times 3$ symmetric matrix-valued functions, where $\Gamma_k^{ij}$ denotes the element of the ith row and the jth column of
$\Gamma_k^{ij}$, $k=1,2$. When the number of players is 300, the curves of state average $x^{N}$ and frozen limit term are in Figure \ref{fig:x^N}. $\epsilon=(\mathbb{E} \int_0^T |z(t)-x^{*(N)}(t)|^{2})^{\frac{1}{2}}$
represents the performance of the decentralized strategies. The figure of $\epsilon$ with respect to the number of players is shown in Figure \ref{fig:epsilon}, which verifies the consistency of mean field approximation.

\begin{figure}[htbp] % [htbp]是位置参数，可选
	\centering % 使图片居中
	\includegraphics[width=0.8\textwidth]{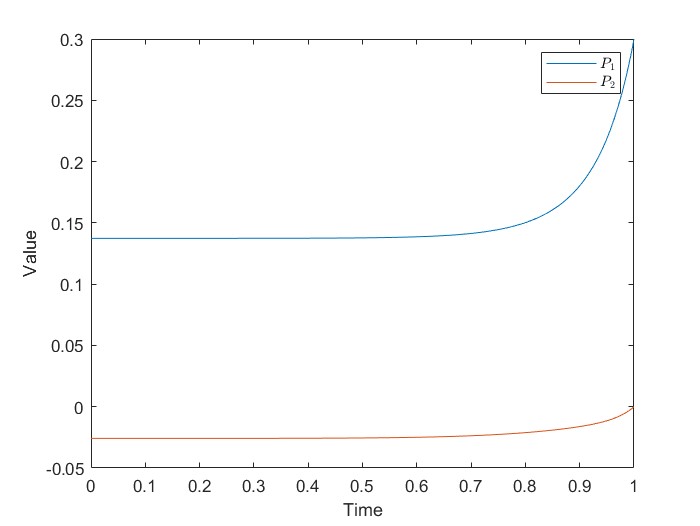} % 插入图片
	\caption{The solution curve of $P_1,P_2$} % 添加图片标题
	\label{fig:P} % 为图片添加标签，方便在文中引用
\end{figure}
\begin{figure}[htbp]
	\centering
	% 第一个子图
	\begin{subfigure}[htbp]{0.8\textwidth}
		\includegraphics[width=\textwidth]{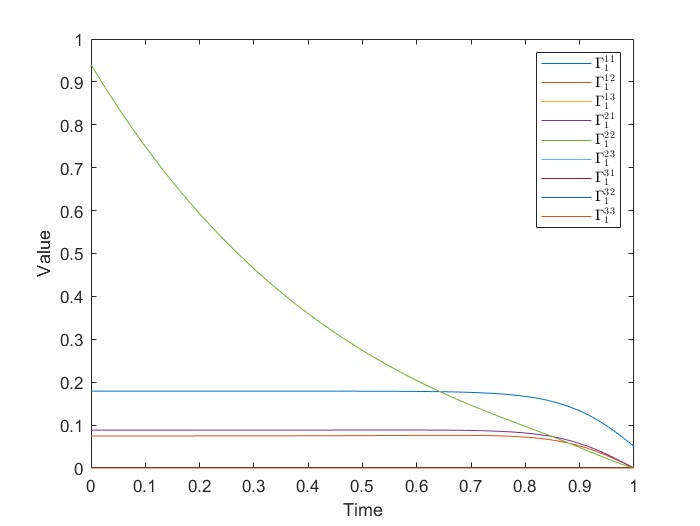}
		\caption{The solution curve of $\Gamma_1$}
		\label{fig:Gamma_1}
	\end{subfigure}
	\hfill % 填充空白，使两个子图之间有一定的间距
	% 第二个子图
	\begin{subfigure}[b]{0.8\textwidth}
		\includegraphics[width=\textwidth]{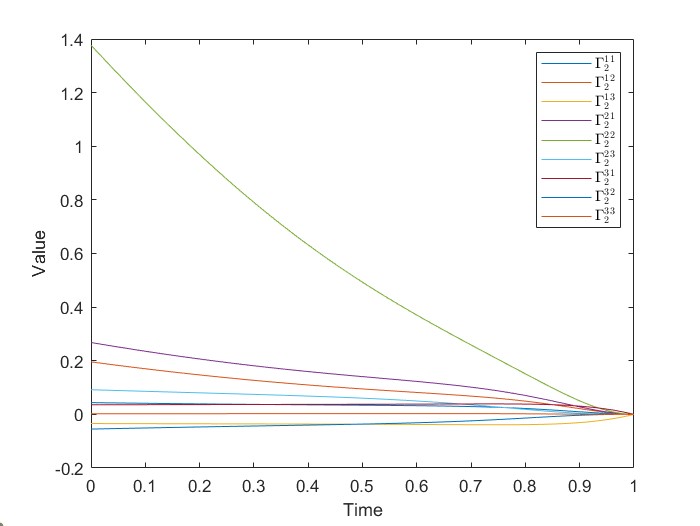}
		\caption{The solution curve of $\Gamma_2$}
		\label{fig:Gamma_2}
	\end{subfigure}
	\caption{}
	\label{fig:Gamma}
\end{figure}

\begin{figure}[htbp] % [htbp]是位置参数，可选
	\centering % 使图片居中
	\includegraphics[width=0.8\textwidth]{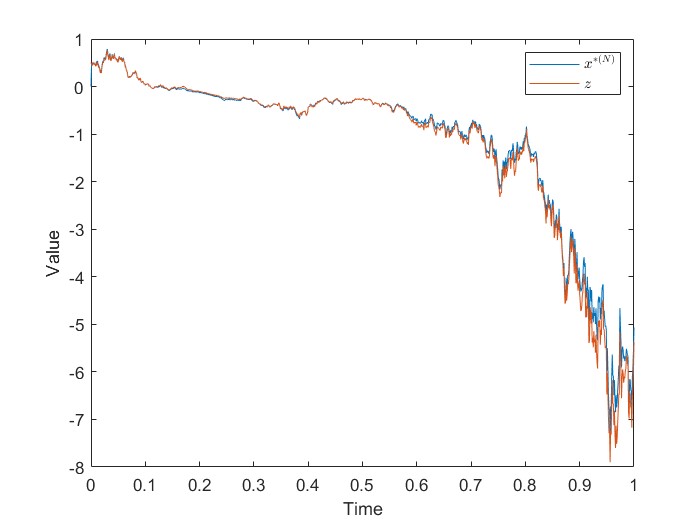} % 插入图片
	\caption{The solution curve of $x^{(N)}$, $z$} % 添加图片标题
	\label{fig:x^N} % 为图片添加标签，方便在文中引用
\end{figure}

\begin{figure}[htbp] % [htbp]是位置参数，可选
	\centering % 使图片居中
	\includegraphics[width=0.8\textwidth]{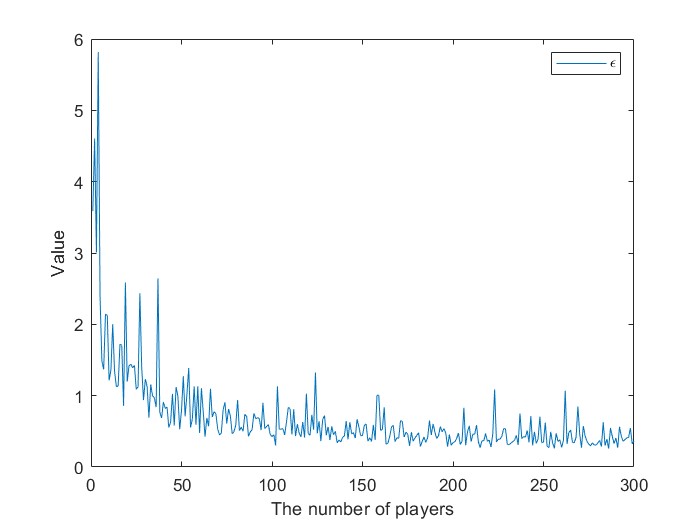} % 插入图片
	\caption{The solution curve of $\epsilon$} % 添加图片标题
	\label{fig:epsilon} % 为图片添加标签，方便在文中引用
\end{figure}

\section{Conclusion}

In this paper, we have considered a linear-quadratic mean field Stackelberg stochastic differential game with partial information and common noise. The state equations of the single leader and multiple followers are general. Decentralized strategies are given by the stochastic maximum principle with partial information and optimal filter technique, and the feedback form of decentralized strategies have been obtained by using some high-dimensional Riccati equations. Finally, the decentralized control strategy obtained has been verified as an $\varepsilon$-Stackelberg-Nash equilibrium of the original game (Problem \ref{problem centralized}).

The general solvability of the high-dimensional Riccati equation (\ref{Gamma2 equation}) is rather challenging. Problems with partial observation are an interesting but difficult ones. We will consider these topics in our future research.

\end{document}